\documentclass[english]{article}

\usepackage{amsfonts}
\usepackage[T1]{fontenc}
\usepackage[latin9]{inputenc}
\usepackage[letterpaper]{geometry}
\usepackage{color}
\usepackage{verbatim}
\usepackage{amsmath}
\usepackage{amssymb}
\usepackage{babel}
\usepackage{amsthm}

\newtheorem{prop}{Proposition}
\newtheorem{theo}{Theorem}
\newtheorem{lemme}{Lemma}
\theoremstyle{remark}
\newtheorem{rem}{Remark}

\newcommand{\R}{\mathbb R}

\newcommand{\B}{\mathbf B}

\newcommand{\Ld}{\mathrm L}
\newcommand{\ep}{\varepsilon}
\newcommand{\HH}{\mathrm H}
\newcommand{\HHc}{\mathrm {H_{curl}}}

\newcommand{\vv}{\mathbf v}

\begin{document}

\title{Nonlinear stability of a Vlasov equation for magnetic plasmas}
\author{Fr\'ed\'erique Charles\thanks{
UPMC-Paris06, CNRS UMR 7598, Laboratoire Jacques-Louis Lions, 4, pl. Jussieu
F75252 Paris cedex 05.} \thanks{Email: charles@ann.jussieu.fr} 
\and Bruno Despr\'es\footnotemark[1]  \thanks{
Email: despres@ann.jussieu.fr} 
\and Beno\^ \i t Perthame\footnotemark[1] \thanks{
INRIA Rocquencourt EPI BANG. Email: benoit.perthame@upmc.fr} \and R\'emi
Sentis\thanks{
CEA Bruy\`eres-le-Chatel. Email: remi.sentis@CEA.FR} }
\date{\today}
\maketitle

\begin{abstract}
The mathematical description of laboratory fusion plasmas produced in
Tokamaks is still challenging. Complete models for electrons and ions, as
Vlasov-Maxwell systems, are computationally too expensive because they take
into account all details and scales of magneto-hydrodynamics. In particular,
for most of the relevant studies, the mass electron is negligible and the
velocity of material waves is much smaller than the speed of light.
Therefore it is useful to understand simplified models. Here we propose and study  one
of those which keeps both the complexity of the Vlasov equation for ions and
the Hall effect in Maxwell's equation. Based on energy dissipation, a
fundamental physical property, we show that the model is nonlinear stable
and consequently prove existence.
\end{abstract}

\bigskip

\noindent \textbf{Key words:} Vlasov equations; Plasma physics; Kinetic
averaging lemma; Maxwell's equations. \\[2mm]
\noindent \textbf{Mathematics Subject Classification:} 35B35, 35L60, 82D10

%########## MODELE ##########################################################
%--------------------------------------------------------------------------------------------------------------------------------------

\section{Introduction}

\label{sec:motivation} 
%--------------------------------------------------------------------------------------------------------------------------------------

To describe the behavior of ions population in hot plasmas, it is very
classical to address, at least at theoretical level, a Vlasov equation
coupled to a non-linear Poisson equation which defines the electrostatic
field. When the magnetic field $\mathbf{B}$ is an external datum, for the
ion distribution function $f(t,x,\mathbf{v})$ (at position $x$ and velocity $
\mathbf{v}$), one addresses the following system 

\begin{eqnarray*}
{\frac{\partial {f}}{\partial {t}}+\mathbf{v}.\nabla f+\frac{\partial }{
\partial \mathbf{v}}.[\left( -{T_{\mathrm{e}}\nabla \ln n_{e}+}\mathbf{v}
\wedge \mathbf{B}\right) f]} &{=}&0,
 \\
- \lambda ^{2}\Delta {\ln {n_{e}}} &{=}&{}\int f(\mathbf{v})d\mathbf{v}{
-n_{e}},
\end{eqnarray*}
where $n_{e}=n_{e}(t,x)$ is the electron density, $T_{\mathrm{e}}$
is the mean electron temperature assumed to be constant and $\lambda $ is
the Debye length (a characteristic constant of the plasma). 
This system is a classical one, indeed the electrostatic field may be approximated by ${-T_{\mathrm{e}}\nabla \ln n_{e}}$ (at least when the electron temperature is constant) and
the Poisson equation is nothing but the Gauss relation applied to this field (cf. \cite{chen} for example)

The mathematical understanding of this kind of kinetic system has made
important progresses with the proof of existence of global weak solutions in the large  of  the
Vlasov-Maxwell system in \cite{DiLi} and of the Vlasov-Poisson system \cite
{LiPe, Pfaf}. However, the mathematical description of laboratory
fusion plasmas (such as those produced in Tokamaks) with this kind of ion
kinetic models is still  a major challenge, in particular to deal with a time scale compatible with the
evolution equation for the magnetic field. As a matter of fact, complete
systems of Vlasov-Maxwell type which account for all scales of
electro-dynamics are not relevant since the velocity of ion waves is much
smaller than the speed of light:
 it is well known  that at the time scale of the ion population it is convenient to neglect the current
of displacement in the Maxwell equations (as in magneto-hydrodynamics models
cf. \cite{lutjens1,lutjens2}), that is to say the electric current $\mathbf{J
}$ is assumed to satisfy

\begin{equation*}
\mathbf{J}=\nabla \wedge \mathbf{B}.
\end{equation*}

The aim of this work is precisely to propose, justify and study a model 
which
couples a kinetic equation for
 the ions and an evolution equation for the
magnetic field (as those used
 in magneto-hydrodynamics). The unknowns are the ion
particle density $f(t,x,\mathbf{v})$, the magnetic field $\mathbf{B}(t,x)$
and the electron density $n_{e}(t,x)$ and they satisfy 
\begin{equation}
\label{model}
\left\{ \begin{array}{ll}
\displaystyle - \lambda ^{2}\Delta \ln {n_{e}={n_{I}-n_{e}}}, & \quad (a) 
\\[3mm] 
 \displaystyle \frac{\partial \mathbf{B}}{\partial {t}} -\nabla \wedge  
  \left( \frac{1}{n_{e}} n_{I}\mathbf{u}_{I}\wedge \mathbf{B}\right) 
  + \nabla \wedge \left( \frac{1}{n_{e}} \mathbf{J} \wedge \mathbf{B} \right)
  + \nabla \wedge \left( \eta \nabla \wedge \mathbf{B} \right) =0, & \quad (b) \\[3mm] 
 \displaystyle \frac{\partial {f}}{\partial {t}}+\mathbf{v} \cdot \nabla f
  +\frac{\partial }{\partial \mathbf{v}}  \cdot 
   \left[ {\left( -\frac{{T_{\mathrm{e}}}}{{n_{e}}}\nabla {n_{e}+}
   \frac{\mathbf{J}-n_{_{I}}\mathbf{u}_{I}}{n_{e}}\wedge 
   \mathbf{\ B}+\mathbf{v}\wedge \mathbf{B}\right) f}\right] =0, & \quad (c)
\\[3mm] 
 \nabla \cdot \mathbf{B}=0, & \quad (d)
\end{array}
\right.  
\end{equation}
where the following notations are used 
\begin{equation}
 \label{defni}
n_{I}(t,x)=\int_{\mathbb{R}^{3}}f(t,x,\mathbf{v})d\mathbf{v}\qquad \quad 
\textrm{(the number density in ions)}, 
\end{equation}
\begin{equation}
\label{defui}
n_{I}(t,x)\mathbf{u}_{I}(t,x)=\int_{\mathbb{R}^{3}}f(t,x,\mathbf{v})
 \mathbf{v}d\mathbf{v}\qquad \textrm{(the macroscopic velocity of ions)}.  
\end{equation}
and $\eta $ denotes a strictly positive bounded function corresponding to
the plasma resistivity.

Notice firstly that the electric field $\mathbf{E}$ is given by the relation 
\begin{equation}
n_{e}\mathbf{E}=-T_{\mathrm{e}}\nabla n_{e}-n_{I}\mathbf{u}_{I}\wedge 
\mathbf{B}+\mathbf{J}\wedge \mathbf{B}+n_{e}\,\eta \nabla \wedge \mathbf{B}
\label{ohm}
\end{equation}
which is one of the classical forms of the generalized Ohm law \cite
{chen,braginskii,blum,freidberg} (it includes the term $\mathbf{J}\wedge 
\mathbf{B}$ related to the Hall effect); so, \eqref{model}-b) is exactly the
Faraday equation ${\frac{\partial \mathbf{B}}{\partial {t}}+\nabla \wedge 
\mathbf{E}}=0.$

 We have chosen here a rescaling, such that the ion mass is set
to $1$ and the electron and ion charge are also set to $1.$ Moreover the
scaling of the ion and electron density is such that the characteristic
value of $n_{I}$ is equal to 1. So with this scaling, the Debye length $\lambda $ is defined by $\lambda ^{2}=T_{\mathrm{e}}\varepsilon _{0}$ where ${\varepsilon }_{0}$ is vacuum dielectric constant and  equation \eqref{model}-a)
reads also as 
$$-{\varepsilon _{0}T_{\mathrm{e}}}\Delta \ln {n_{e}} =n_{I}- n_e  $$
  which is Gauss relation applied to 
 $-T_{e}\nabla \ln {n_{e}},$ the dominant term in  the electric field.
  In the rest of the paper, we use the following notations 
\begin{equation}
\label{ohmo}
\mathbf{E}^{0}(t,x)=-T_{\mathrm{e}}\frac{\nabla n_{e}}{n_{e}}+ 
\frac{1}{n_{e}} (\mathbf{J}-n_{_{I}}\mathbf{u}_{I})\wedge \mathbf{B}
\qquad 
\mbox{ and }
\qquad
\mathbf{F}(t,x,\mathbf{v})=\mathbf{E}^{0}+\mathbf{v}\wedge {\mathbf{B}}.
\end{equation}
         Here
   $\mathbf{F}$ is the kinetic force field which appears in equation %
\eqref{model}-c).

\begin{rem} It is fundamental to notice that the Debye length is a
small quantity in many situations; it may be proved that in the limit 
$\lambda \rightarrow 0$ the solution $n_{e}$ to the non-linear Poisson
equation satisfies ${n_{e}}\rightarrow {n_{I}}$ and many studies have
analyzed this singular limit in different contexts \cite{breGS,DegondQN, DHKquasi}. Then, if we make the approximation ${n_{e}}\simeq {n_{I},}$
the Ohm law (\ref{ohm}) leads to the classical formula 
\begin{equation*}
\mathbf{E}+\mathbf{u}_{I}\wedge \mathbf{B}=-T_{\mathrm{e}}\nabla \ln {n_{I}}+
\frac{1}{n_{I}}(\nabla \wedge \mathbf{B)}\wedge \mathbf{B}+\eta \nabla
\wedge \mathbf{B}
\end{equation*}
and from (\ref{model}), we recover formerly the following system 
\begin{eqnarray*}
\frac{\partial \mathbf{B}}{\partial {t}}-\nabla \wedge \left( \mathbf{u}
_{I}\wedge \mathbf{B}\right) +\nabla \wedge \left( \frac{1}{n_{I}}(\nabla
\wedge \mathbf{B)}\wedge \mathbf{B})\right) +\nabla \wedge \left( \eta
\nabla \wedge \mathbf{B}\right)  &=&0, \\
\frac{\partial {f}}{\partial {t}}+\mathbf{v}.\nabla f+\frac{\partial }{
\partial \mathbf{v}}\left[ {(}\mathbf{E}^{\lim }+{\mathbf{v}\wedge \mathbf{B}
})f\right]  &=&0, \\
\qquad n_{I}(\mathbf{E}^{\lim }+\mathbf{u}_{I}\wedge \mathbf{B}) &=&-T_{
\mathrm{e}}\nabla n_{I}+(\nabla \wedge \mathbf{B})\wedge \mathbf{B}.
\end{eqnarray*}
Despite its apparently simpler form  than the system \eqref{model}, it is not
used in physical literature, up to our knowledge. As a matter of fact even
if the magnetic field is neglected, the equation 
\begin{equation*}
\frac{\partial {f}}{\partial {t}}+\mathbf{v}.\nabla f+\frac{\partial }{
\partial \mathbf{v}}\left[ {(}\mathbf{E}^{\lim }+{\mathbf{v}\wedge \mathbf{B}
})f\right] =0,\ \quad n_{I}\mathbf{E}^{\lim }=-T_{\mathrm{e}}\nabla n_{I}
\end{equation*}
is probably ill-posed from a mathematical point of view (nevertheless 
\cite{DHKquasi} has noticed that using some arguments of the paper \cite{MouVI},
one may find a weak solution on a small time interval). Moreover it would be
hard to give a meaning of the quantity 
$\mathbf{u}_{I}=\int_{\mathbb{R}^{3}}f(\mathbf{v})\mathbf{v}d\mathbf{v/}
 \int_{\mathbb{R}^{3}}f(\mathbf{v})d\mathbf{v}$ 
which appears in the magnetic equation (it is hardly possible to derive a bound
of $n_{I}$ away from below). 
Notice that this kind of model has been addressed in  \cite{GHN} and \cite{AD} 
but the framework is a little different since a velocity diffusion term
(of Fokker-Planck type) is accountered for  in the kinetic equation.
%(in this last reference there is a Fokker-Plank diffusion term in the kinetic equation).
\end{rem}

Back to \eqref{model}, from a mathematical point of view, one can show that $n_{e}$ is
naturally bounded away from zero due to the nonlinear elliptic equation 
\eqref{model}-a and the quantity $\frac{n_{I}\mathbf{\ }}{n_{e}}\mathbf{u}
_{I}=n_{e}^{-1}{\int_{\mathbb{R}^{3}}}f(\mathbf{v})\mathbf{v}d\mathbf{v}$
may be defined in some Lebesgue space.

Our study is aiming at the analysis of model \eqref{model} in 
a bounded domain; 
therefore we have to  specify the boundary and initial
conditions.

%---------------

\paragraph{Boundary conditions}

Let $\Omega $ be a bounded domain of $\mathbb{R}^{3}$ of class $\mathcal{C}
^{1,1}$. We consider the model \eqref{model} for $\mathbf{v}\in \mathbb{R}
^{3}$, $x\in \Omega $ and $t\geq 0$, with the following boundary conditions 
\begin{equation}
\left\{ \begin{aligned} & \mathbf{n}_x\cdot \nabla n_e(t,x) =0 \quad & x\in
\partial\Omega, & \quad \quad(a)\\ & \mathbf{n}_x \wedge \mathbf
B(t,x)=0\quad & x\in \partial\Omega, & \quad \quad(b)\\ & f (x, \mathbf v -2
(\mathbf v \cdot \mathbf{n}_x)\mathbf{n}_x) = f (x, \mathbf v), \quad & \mathbf v \in \mathbb R^3, \quad x\in
\partial\Omega. & \quad \quad(c)
\end{aligned}\right.   \label{BC}
\end{equation}
The zero flux boundary conditions \eqref{BC}-a) on electrons enforces in 
\eqref{model}-a) the global neutrality of the plasma 
\begin{equation*}
\int_{\Omega }n_{e}dx=\int_{\Omega }n_{i}dx.
\end{equation*}
The condition (b) on the magnetic field is the simplest one and a more realistic condition is discussed in section \ref{section:binhom}. 
Our result will be
still true with more realistic boundary condition, of impedance type for
example, or also a non homogeneous boundary condition such as 
\begin{equation*}
\mathbf{n}_{x}\wedge \mathbf{B}(t,x)=\mathbf{n}_{x}\wedge \mathbf{B}_{
\mathrm{imp}},\quad x\in \partial \Omega ,
\end{equation*}%
which is much more relevant in the context of confined plasmas in Tokamaks.
%This boundary condition is discussed later in section \ref{section:binhom}.
Condition \eqref{BC}-c) is the so-called specular reflection; it implies
the no-slip boundary condition
\begin{equation}
\mathbf{u}_{I}(t,x)\cdot \mathbf{n}_{x}=0,\quad x\in \partial \Omega .
\label{BCui}
\end{equation}

Our main result is the following.
%----------------------------------

\begin{theo}[Nonlinear stability]
Assume  \eqref{Hfin}  and that the initial datum satisfy \eqref{H1}--\eqref{H6}.  
From any bounded family of solutions $f^{\varepsilon },\mathbf{B}
^{\varepsilon },n_{e}^{\varepsilon }$ to \eqref{model}, \eqref{BC} we can
extract a subsequence that converges to a weak solution with finite energy.
Moreover, we have for all $T>0$,
\begin{equation*}
\mathbf{B}^{\varepsilon }\rightarrow \mathbf{B}\in \mathrm{L}^{\infty }
 \big(0,T;\mathrm{L}^{2}(\Omega )\big)\cap \mathrm{L}^{2}\big(0,T;\mathrm{L}
^{q}(\Omega )\big)\quad \textrm{ strongly},\qquad
q<6,
\end{equation*}
\begin{equation*}
f^{\varepsilon }\rightarrow f\in \mathrm{L}^{\infty }\big(0,T;\mathrm{L}
^{1}\cap \mathrm{L}^{\infty }(\Omega \times \mathbb{R}^{3}))\big)\quad 
\textrm{weakly},
\end{equation*}
\begin{equation*}
\mathbf{J}^{\varepsilon }=\nabla \wedge \mathbf{B}^{\varepsilon }
\rightarrow \mathbf{J}\quad \mathrm{ in }\; \mathrm{L}^{2}((0,T)\times \Omega )\quad 
\textrm{weakly},
\end{equation*}
\begin{equation*}
n_{e}^{\varepsilon }\rightarrow n_{e}\quad \textrm{ in }\mathrm{L}
^{q}((0,T)\times \Omega )\textrm{, strongly and for some constants}
\quad 0<K_{-}\leq n_{e}\leq K_{+},
\end{equation*}
\begin{equation*}
n_{I}^{\varepsilon }\mathbf{u}_{I}^{\varepsilon }\rightarrow n_{I}\mathbf{u}
_{I}\in (\mathrm{L}^{q}(0,T\times \Omega ))^{3} \quad \mathrm{strongly},
\end{equation*}
for all $q$ such that $1\leq q<5/4$.
\label{main}
\end{theo}

%----------------------------------

The first tool towards a mathematical analysis is to state the energy balance
for the full system; this is performed in section~\ref{sec:prel}. It indicates that 
the magnetic resistivity in equation (\ref{model})-a) is useful (if not necessary) so as
to control the current $\mathbf{J}$ and subsequently the  term $\mathbf{J}
\wedge \mathbf{B}$ describing to the Hall effect. Also the classical
tools for elliptic equation with Neumann boundary conditions will lead to
control $n_{e}$ away from below. But more technical ingredients (namely a
kinetic averaging lemma in a $L^{q}$ space) are used in the proof in order
to pass to the limit in the nonlinearities; this is the purpose of section~\ref{sec:stab} 
and of several appendices. We come back on more realistic boundary 
conditions in section~\ref{section:binhom} so as to include confinement. 
We also 
detail a constructive splitting  procedure which can be used
to design  an approximate solution with the same energy
law.
%on increasing total energy.
%############### PRELIMINARIES #################################################

\section{Preliminaries}
\label{sec:prel} 
%---------------------------------------------------------------------------------------------------------------------------------------------

%--------------------------------------------------------------------------------------------

%%%%%%%%%%%%%%%%%%%%%%%%%%%%%%%%%%%

The first and major ingredient in our approach to nonlinear stability is the energy balance. We state it here
 together with considerations on the physics sustaining the model.

%---------------------------------------------------------------------------

\subsection{The energy balance}

%---------------------------------------------------------------------------
%%%%%%%%%%%%%%%%%%%%%%%%%%

We now turn to the study of energy dissipation for the system (\ref{model})
with homogeneous boundary conditions \eqref{BC}. We introduce the ions
kinetic energy and the magnetic energy 
\begin{equation*}
\mathcal{E}_{I}(t)=\frac{1}{2}\int_{\Omega }{\int_{\mathbb{R}^{3}}f(t,x,
\mathbf{v})|\mathbf{v}}|^{2}{d\mathbf{v}}dx,\qquad \quad \mathcal{E}_{m}(t)=
\frac{1}{2}\int_{\Omega }|\mathbf{B}(t,x)|^{2}dx,
\end{equation*}
as well as the free energy 
$\displaystyle \int_{\Omega }(n_{e}(t)\ln n_{e}(t)-n_{e}(t)+1)dx$ 
and the total energy which reads as 
\begin{equation}
\label{def_energy}
\mathcal{E}_{tot}(t)=\mathcal{E}_{I}(t)+\mathcal{E}_{m}(t)+\frac{\lambda ^{2}
}{2}\int_{\Omega }\left\vert \nabla \ln n_{e}(t)\right\vert
^{2}dx+\int_{\Omega }(n_{e}(t)\ln n_{e}(t)-n_{e}(t)+1)dx.
\end{equation}
Notice that the integral 
$\displaystyle \frac{\lambda ^{2}}{2}\int_{\Omega }\left\vert \nabla \ln n_{e}(t)\right\vert ^{2}dx$ 
corresponds to the electrostatic energy.
We now establish the

\begin{prop} [Energy dissipation] Classical solutions to (\ref{model}), \eqref{BC} satisfy
the energy dissipation relation 
\begin{equation}
\frac{d}{dt}\left[ \mathcal{E}_{I}+\mathcal{E}_{m}+\frac{\lambda ^{2}}{2}%
\int_{\Omega }\left\vert \nabla \left( \ln n_{e}\right) \right\vert
^{2}dx+\int_{\Omega }\left( n_{e}\ln n_{e}-n_{e}+1\right) dx\right]
=-\int_{\Omega }\eta \Bigl|\nabla \wedge \mathbf{B}\Bigr|^{2}dx.
\label{estimationNRJ}
\end{equation}
\label{propenergy}
\end{prop}

\begin{proof}
Notice that, as usual,  the free divergence condition
(\ref{model})-(d) is not needed for the energy identity.
We first consider the ion kinetic energy. Using the specular reflection
condition, we compute
\begin{eqnarray*}
\frac{d}{dt}\mathcal{E}_{I} 
&=&\int_{\Omega }\int_{\mathbb{R}^{3}}\mathbf{F}(t,x,\mathbf{v})\cdot \mathbf{v}d\mathbf{v} \\
&=&\int_{\Omega }\int_{\mathbb{R}^{3}}
   \Biggl[\biggl(\frac{1}{n_{e}}\left(\nabla \wedge \mathbf{B}\right) \wedge \mathbf{B}
         +(\mathbf{v}-\frac{n_{I}}{n_{e}}\mathbf{u}_{I})\wedge \mathbf{B}-T_{\mathrm{e}}\nabla
        \left( \ln n_{e}\right) \biggr)f\Biggr]\cdot \mathbf{v}d\mathbf{v}dx \\
&=&\int_{\Omega }\left( \frac{n_{I}}{n_{e}}\mathbf{J}\wedge \mathbf{B}
      -n_{I}T_{\mathrm{e}}\nabla \left( \ln n_{e}\right) \right) \cdot \mathbf{u}_{I}dx.
\end{eqnarray*}
Now, we turn to the magnetic energy and first recall the definition of $
\mathbf{E}$ in \eqref{ohm}. Thanks to a classical tensorial identity, we
find 
\begin{eqnarray*}
\frac{d}{dt}\mathcal{E}_{m}
 &=&\int_{\Omega }\frac{\partial \mathbf{B}}{\partial t}\cdot \mathbf{B}dx
  =-\int_{\Omega }\nabla \wedge \mathbf{E}\cdot  \mathbf{B}dx
  =-\int_{\Omega }[\mathbf{E}\cdot \nabla \wedge \mathbf{B}
    +\nabla \cdot (\mathbf{E}\wedge \mathbf{B})]dx \\
 &=&-\int_{\Omega }\left[ -\frac{n_{I}}{n_{e}}\mathbf{u}_{I}\wedge \mathbf{B}
    +\eta \nabla \wedge \mathbf{B}\right] \cdot \left( \nabla \wedge \mathbf{B}
   \right) dx-\int_{\partial \Omega }\mathbf{n}_{x} \cdot(\mathbf{E\wedge B})dx.
\end{eqnarray*}
According to the boundary condition \eqref{BC}-(b), we have
 $\mathbf{n}_{x}\cdot( \mathbf{E\wedge B}) =-\mathbf{E} \cdot (\mathbf{n}_{x}\mathbf{\wedge B})=0,$ 
then the boundary integral is zero. Now, thanks to identity $\mathbf{J}\cdot
\left( \mathbf{u}_{I}\wedge \mathbf{B}\right) 
=-\mathbf{u}_{I}\cdot (\mathbf{
J}\wedge \mathbf{B}),$ we get
\begin{equation*}
\frac{d}{dt}\mathcal{E}_{m}=-\int_{\Omega }\frac{n_{I}}{n_{e}}\mathbf{u}
_{I}\cdot (\mathbf{J}\wedge \mathbf{B})dx-\int_{\mathbb{R}^{3}}\eta \Bigl|
\nabla \wedge \mathbf{B}\Bigr|^{2}dx.
\end{equation*}
At this stage, we have obtained 
\begin{equation} 
\label {eq:ttt3}
\begin{array}{llll}
 \displaystyle \frac{d}{dt}[\mathcal{E}_{i}+\mathcal{E}_{m}] 
&=&
 \displaystyle 
\int_{\Omega }\left[ (
\mathbf{J}\wedge \mathbf{B})-n_{I}\nabla \left( \ln n_{e}\right) \right]\cdot
\mathbf{u}_{I}dx-\int_{\Omega }\mathbf{u}_{I} \cdot (\mathbf{J}\wedge \mathbf{B}
)dx-\int_{\Omega }\eta \Bigl|\nabla \wedge B\Bigr|^{2}dx \\[3mm]
&=&-\int_{\Omega }n_{I}\nabla \left( \ln n_{e}\right) \cdot \mathbf{u}
_{I}dx-\int_{\Omega }\eta \Bigl|\nabla \wedge B\Bigr|^{2}dx.
\end{array}
\end{equation}
To continue, we recall the conservation of mass on $n_{I}$ 
\begin{equation*}
\frac{\partial n_{I}}{\partial t}+\nabla \cdot (n_{I}\mathbf{u}_{I})=0.
\end{equation*}
Then, we may use the boundary condition \eqref{BCui} to obtain 
\begin{eqnarray*}
-\int_{\Omega }n_{I}\nabla \left( \ln n_{e}\right) \cdot \mathbf{u}_{I}dx
&=&\int_{\Omega }(\ln n_{e})\nabla \cdot \left( n_{I}\mathbf{u}_{I}\right) dx
\\
&=&-\int_{\Omega }(\ln n_{e})\frac{\partial n_{I}}{\partial t}dx \\
&=&-\int_{\Omega }(\ln n_{e})\left[ \frac{\partial n_{e}}{\partial t}
-\lambda ^{2}\frac{\partial }{\partial t}\Delta \left( \ln n_{e}\right) 
\right] dx \\
&=&-\frac{d}{dt}\left[ \int_{\Omega }\left( n_{e}(\ln n_{e})-n_{e}\right) dx
\right] -\lambda ^{2}\frac{d}{dt}\left[ \int_{\Omega }\frac{\left\vert
\nabla (\ln n_{e})\right\vert ^{2}}{2}dx\right] 
\end{eqnarray*}
where we have used equation \eqref{model}-(a) once differentiated in time. 
\newline
Altogether, we obtain the relation \eqref{estimationNRJ}.
\end{proof}

%--------------------------------------------------------------------------------------------

\subsection{Physical considerations}

Let first focus on the momentum balance. To find it, we multiply the kinetic
equation (\ref{model})-c) by $\mathbf{v}$, integrate over $\mathbf{R}^{3}$
and define as usual \cite{CIP} the ion pressure tensor $\mathrm{P}_{I}$
through 
\begin{equation*}
\int \nabla \cdot (\mathbf{v}\otimes \mathbf{v}f(\mathbf{v}))d\mathbf{v}
=\nabla \cdot (n_{I}\mathbf{u}_{I}\otimes \mathbf{u}_{I})+\nabla \cdot 
\mathrm{P}_{I}.
\end{equation*}
Then, we get 
\begin{equation*}
\frac{\partial ({n}_{I}\mathbf{u}_{I})}{\partial {t}}+\nabla \cdot (n_{I}
\mathbf{u}_{I}\otimes \mathbf{u}_{I})+\nabla \cdot \mathrm{P}_{I}+T_{\mathrm{
e}}\frac{{\nabla {n_{e}}}}{n_{e}}n_{I}+\frac{n_{_{I}}}{n_{e}}n_{_{I}}\mathbf{
u}_{I}\wedge \mathbf{B}-\frac{n_{_{I}}}{n_{e}}\mathbf{J}\wedge \mathbf{B}
-n_{_{I}}\mathbf{u}_{I}\wedge {\mathbf{B}}=0,
\end{equation*}
that is to say 
\begin{equation*}
\frac{\partial ({n}_{I}\mathbf{u}_{I})}{\partial {t}}+\nabla \cdot (n_{I}
\mathbf{u}_{I}\otimes \mathbf{u}_{I})+\nabla \cdot \mathrm{P}_{I}+T_{\mathrm{e}}{
\nabla {n_{e}}}-\mathbf{J}\wedge \mathbf{B}=\frac{(n_{e}-n_{I})}{n_{e}}
\left[ T_{\mathrm{e}}\nabla n_{e}+(n_{I}\mathbf{u}_{I}-\mathbf{J})\wedge 
\mathbf{B}\right].
\end{equation*} 
On the left hand side one recognizes the classical momentum equation which
appears in the MHD modeling with the magnetic pressure tensor. Let us stress
that with the electric field $\mathbf{E}^{0}$ introduced in \eqref{ohmo} and
according to (\ref{model})-a), the r.h.s. term reads as follows 
\begin{equation*}
\varepsilon_{0} \left[ \nabla \cdot \mathbf{E}^{0}+\nabla \cdot \left( \frac{n_{I}
\mathbf{u}_{I}-\mathbf{J}}{n_{e}}\wedge \mathbf{B}\right) \right]
\mathbf{E}^{0}=\varepsilon_{0}{\nabla \cdot }\left( \frac{\overline{\overline{1}}}{2}
|\mathbf{E}^{0}|^{2}-\mathbf{E}^{0}\otimes \mathbf{E}^{0}\right) +\lambda ^{2}
\mathbf{E}^{0}  \nabla \cdot \left( \frac{n_{I}\mathbf{u}_{I}-\mathbf{J}}{n_{e}\,T_{
\mathrm{e}}}\wedge \mathbf{B}\right).
\end{equation*}
Thus we may see that this equation is in a conservative form only when the
Debye length vanishes. This is due to the asymptotic expansion motivating (
\ref{model}); we have chosen to keep energy dissipation and smoothness for $
n_{e}$ to the expense of momentum balance. 
In order to ensure exact conservation momentum, a possible route is to use a
modified Poisson equation as follows 
\begin{equation*}
-\lambda ^{2}\Delta \ln {n_{e}}=n_{I}-n_{e}-\lambda ^{2}\nabla \cdot \left( 
\frac{\mathbf{J}-n_{I}\mathbf{u}_{I}}{n_{e}T_{\mathrm{e}}}\wedge \mathbf{B}%
\right) 
\end{equation*}%
which corresponds to an approximation at the same order but which looses
nice properties on $n_{e}$.

\begin{rem} [Electron momentum balance] We see that $\mathbf{E}^{0}$ is equal to the electric field 
$\mathbf{E}$ up to the resistive part $\eta \nabla \wedge \mathbf{B}$. Let
us stress that 
$(\nabla \wedge \mathbf{B)-}n_{I}\mathbf{u}_{I}=\mathbf{J} -n_{I}\mathbf{u}_{I}$ 
is the electron contribution to the electric current and \eqref{ohm}
 may be interpreted as the electron momentum balance equation  when using
the massless electron approximation. 
\end{rem}

\begin{rem} [Friction] In order  to account for the friction between ions and
electrons in the kinetic equation, we could add a friction term to equation (%
\ref{model})-c) and arrive to 
\begin{equation*}
\frac{\partial {f}}{\partial {t}}+\mathbf{v}.\nabla f+\frac{\partial }{%
\partial \mathbf{v}}.(\mathbf{F}f)=-{\frac{\partial }{\partial \mathbf{v}}(}%
\eta \mathbf{J}f)
\end{equation*}%
(recall that the coefficient $\eta $ is related to the collision frequency
of electrons against ions). Also, for the sake of compatibility, instead of
equation (\ref{model})-a), the Poisson equation would read 
\[
\lambda^{2}\big(-\Delta \ln n_{e}+\frac{1}{T_{\mathrm{e}}}\nabla (\eta \mathbf{J})\big)=n_{I}-n_{e}.
\]
\end{rem}

%############# Weak stability #########################################
%----------------------------------------------------------------------------------------------------------------------

\section{Proof of the stability theorem}
\label{sec:stab} 
%----------------------------------------------------------------------------------------------------------------------

We give a family of initial conditions that satisfy 
\begin{equation}
f^{\varepsilon }(0,x,\mathbf{v})=f^{in,\varepsilon }(x,\mathbf{v})\geq
0,\qquad f^{in,\varepsilon }\textrm{ is bounded in }\mathrm{L}^{1}\cap \mathrm{
L}^{\infty }(\Omega \times \mathbb{R}^{3}),  \label{H1}
\end{equation}
\begin{equation}
\mathbf{B}^{\varepsilon }(0,x)=\mathbf{B}^{in,\varepsilon }(x),\qquad \nabla
_{x}\cdot \mathbf{B}^{in,\varepsilon }=0.  \label{H4}
\end{equation}
From the Lemma \ref{ni53} in Appendix \ref{ap:moments} and the explanations
in Appendix \ref{ap:elliptic} we see that if $\mathcal{E}_{I}(0)$ is
uniformly bounded then the free energy at initial time 
$\displaystyle{\ \int_{\Omega }\left\vert \nabla _{x}\left( \ln n_{e}^{\varepsilon
}(0,x)\right) \right\vert ^{2}dx+\int_{\Omega }n_{e}^{\varepsilon }(0,x)\ln
n_{e}^{\varepsilon }(0,x)dx}$ is also uniformly bounded. Then we only need
to assume that 
\begin{equation}
\sup_{0<\varepsilon \leq 1}\left[ \mathcal{E}_{I}(0)+\mathcal{E}_{m}(0)%
\right] <\infty  \label{H6}
\end{equation}
to have a total energy uniformly bounded for all time.
Moreover we assume a control on resistivity as 
\begin{equation}
0<\eta _{\mathrm{min}}\leq \eta \in \mathrm{L}^{\infty }(\Omega ).
\label{Hfin}
\end{equation}
Associated with such initial datum, we consider a sequence of strong solutions $f^{\varepsilon }$, $%
\mathbf{B}^{\varepsilon }$, and $n_{e}^{\varepsilon }$ of \eqref{model} with
boundary conditions \eqref{BC}.

%-------------------------------------------------------------
\subsection{A priori bounds}
\label{ssec:ape} 
%-------------------------------------------------------------

%-------------------------------------------------------------

Before we begin the proof of Theorem \ref{main}, we recall several general a
priori estimates that are used. The energy bound \eqref{H6} is of course at
the heart of our analysis because it allows us to deduce directly from the
estimate \eqref{estimationNRJ} the

\begin{prop}[Bounds derived from energy]
Under the assumptions  \eqref{H1}--\eqref{Hfin}, the sequences $f ^{\varepsilon}$, $
\mathbf{B}^{\varepsilon}$, and $n_e^{\varepsilon}$ satisfy 
\begin{equation}  \label{estf}
\sup_{0<\varepsilon\leq 1} \sup_{t\in[0,\infty]} \int_{\mathbb{R}^{3}}\int_{
\Omega}f_{i}^{\varepsilon}(t,x,\mathbf{v})|\mathbf{v}|^{2}d\mathbf{v }dx
<\infty,
\end{equation}
\begin{equation}  \label{estb}
\sup_{0<\varepsilon\leq 1} \sup_{t\in[0,\infty]} \int_{\Omega}| \mathbf{B}
^{\varepsilon}(t,x)|^{2}dx <\infty,
\end{equation}
\begin{equation}  \label{propnep}
\sup_{0<\varepsilon\leq 1} \sup_{t\in[0,\infty]} \int_{\Omega} \left(
\left|\nabla_{x}\left(\ln\left|n_{e}^{\varepsilon}\right|\right)\right|^{2}
+ n_{e}^{\varepsilon}\ln\left|n_{e}^{\varepsilon}\right|+
n_{e}^{\varepsilon} \right)dx <\infty.
\end{equation}
\end{prop}

Indeed, for \eqref{propnep}, we use the fact that the map $x\in\mathbb{R }
\mapsto x\ln(x)-x$ is bound from below and thus all the terms in \eqref{estimationNRJ} are bounded by
the initial energy plus a constant. 
\newline
This provides a control on the kinetic energy for the ion distribution that
is essential to obtain further a priori bounds

\begin{prop}[Estimates for the kinetic densities]
\label{propkinetic} For strong solutions one has $f ^{\varepsilon} \in 
\mathrm{L}^\infty \big(0, \infty; \mathrm{L}^1\cap \mathrm{L}%
^\infty(\Omega\times \mathbb{R}^3)\big)$ and 
\begin{equation*}
\| f ^{\varepsilon}(t,\cdot,\cdot) \|_{p} = \|
f^{in,\varepsilon}(\cdot,\cdot)\|_{p} \quad \forall t \geq 0 , \; \forall
p\in [1,\infty],
\end{equation*}
\begin{equation*}
\| n_{I}^{\varepsilon}(t,\cdot)\|_{5/3} \leq C \| f^{in,\varepsilon}
(\cdot,\cdot) \|_{\infty}^{2/5} \left(\int_{\Omega} \int_{\mathbb{R}^3} f
(t,x,\mathbf{v})|\mathbf{v}|^2d\mathbf{v }dx \right)^{3/5} ,
\end{equation*}
\begin{equation}  \label{current}
\| n_{I}^{\varepsilon} \mathbf{u}_I^{\varepsilon}(t,\cdot)\|_{5/4}\leq C \|
f^{in,\varepsilon} (\cdot,\cdot) \|_{\infty}^{1/5} \left(\int_{\Omega} \int_{%
\mathbb{R}^3} f (t,x,\mathbf{v})|\mathbf{v}|^2d\mathbf{v }dx \right)^{4/5}.
\end{equation}
\end{prop}

The bounds of $f^\ep$ follow immediately from the observation that $\mathrm{%
div}_\vv F=0$. The bounds on $n_I^\ep$ and $n_{I}^{\varepsilon} \mathbf{u}%
_I^{\varepsilon}$ follow from interpolation inequalities that we recall in
the Appendix \ref{ap:moments}. \newline

Next we use the $L^{5/3}$ integrability of $n_I$ to obtain

\begin{prop}[Estimates on the electron density]
\label{propnietne} The electron density satisfies for some constants $K_{+}
>0$, $K_{-}>0$ (depending on the bounds stated as now) 
\begin{equation}  \label{elliptic1}
0 < K_- \leq n_e (t,x) \leq K_+ .
\end{equation}
\end{prop}

This is an easy consequence of elliptic regularity and we give a proof in
the Appendix \ref{ap:elliptic}, it is just a combination of Lemma \ref{ni53}
and Lemma \ref{encadrementne}. \newline

We now come to the magnetic field. We have the following Proposition :

\begin{prop}[Estimate on the magnetic field]
\label{propmagnetic} The family of magnetic fields satisfies for a uniform constant $C$ 
\begin{equation*}
\sup_{0< \varepsilon \leq 1} \int_0^\infty \int_\Omega \Bigl|
\nabla_{x}\wedge \mathbf{B}^\ep \Bigr|^{2}dx dt \leq C ,
\end{equation*}
\begin{equation}  \label{bblsix}
\sup_{0< \varepsilon \leq 1} \| \mathbf{B}^{\varepsilon}\|_{L^2\big(
0,\infty; L^6(\Omega)\big)} \leq C.
\end{equation}
\end{prop}

\begin{proof}
 The control in $\Ld^2(0,\infty; \HHc(\Omega))$, follows from
 the energy dissipation in Proposition \ref{propenergy}. Then thanks to Theorem \ref{theoBernardi}, the family  $\B^{\ep}$ is bounded in $\Ld^2(0,\infty;\HH^1(\Omega)^3)$, and then, thanks to the Sobolev injections, the family $\B^{\varepsilon}$ is bounded in $L^2(0,\infty; \Ld^6(\Omega)^3 \big)$. We refer to Appendix \ref{ap:magnetic} for precise statements and references. 
\end{proof}

Our last task is to show that we have enough bounds to define the electric
and force fields. The four terms that compose the electric field 
\begin{equation*}
\mathbf{E^{\varepsilon}}=-T_{\mathrm{e}}\nabla \ln n_{e}^{\varepsilon} - 
\frac{n_{I}^{\varepsilon} \mathbf{u}_{I}^{\varepsilon} } {n_e^{\varepsilon}}
\wedge \mathbf{B^{\varepsilon}} + \frac{\mathbf{J^{\varepsilon}}\wedge 
\mathbf{B^{\varepsilon}}}{n_e^{\varepsilon}} + \eta \nabla \wedge \mathbf{
B^{\varepsilon}}
\end{equation*}
are respectively uniformly bounded in the spaces

\begin{itemize}
\item $T_{\mathrm{e}}\nabla \ln n_{e}^{\varepsilon} \in L^\infty_t(L^2_x)$
(energy inequality),
\item $\displaystyle \frac{n_{I}^{\varepsilon} \mathbf{u}_{I}^{\varepsilon} 
} {n_e^{\varepsilon}} \wedge \mathbf{B^{\varepsilon}} \in L^2_t (\mathrm{L}
^{30/29}_x);$ this follows from \eqref{current}, \eqref{elliptic1} and \eqref{bblsix} and the
fact that, according to H\"{o}lder inequality, we get the bound 
$ \displaystyle 
\int_{0}^{T} [ \left\Vert AB\right\Vert _{30/29 } ]^{2} dt \leq
\int_{0}^{T}[\left\Vert A\right\Vert _{5/4}\left\Vert B\right\Vert
_{6}]^{2}dt\leq \sup_{t}\left\Vert A\right\Vert
_{5/4}^{2}\int_{0}^{T}[\left\Vert B\right\Vert _{6}]^{2}dt $.
\item $\displaystyle  \frac{\mathbf{J^{\varepsilon}}\wedge \mathbf{%
B^{\varepsilon}}}{n_e^{\varepsilon}} \in \mathrm{L}^1_t(L^{3/2}_x) \cap 
\mathrm{L}_t^{2}(L_x^1)$ from the H\"older inequality with \eqref{elliptic1}  and the a priori
estimates in Proposition~\ref{propmagnetic}. By interpolation, we also find
that 
\begin{equation*}
\displaystyle \frac{\mathbf{J^{\varepsilon}}\wedge \mathbf{B^{\varepsilon}}}{%
n_e^{\varepsilon}} \in \mathrm{L}^r_t(L^{p}_x), \qquad 1\leq r \leq 2, \quad
\frac 1 p = \frac 43 - \frac{2}{3r}.
\end{equation*}
In particular $\displaystyle  \frac{\mathbf{J^{\varepsilon}}\wedge \mathbf{%
B^{\varepsilon}}}{n_e^{\varepsilon}} \in \mathrm{L}^{20/11}_t(L^{30/29}_x)$.

\item $\eta \nabla \wedge \mathbf{B^{\varepsilon}} \in L^2_{t,x}$ (energy
dissipation).
\end{itemize}

Therefore it is well defined and the four terms will have a weak limit in Lebesgue spaces.
\\

For the kinetic force field, 
\begin{equation*}
\mathbf{F}^{\varepsilon}(t,x,v):=- T_{\mathrm{e}} \nabla \ln {%
n_{e}^{\varepsilon}} + \left(\mathbf{v}-\frac{n_{I}^{\varepsilon}}{%
n_e^{\varepsilon}}\mathbf{u}_I^{\varepsilon} \right) \wedge{\mathbf{B}}%
^{\varepsilon} + \frac{ \mathbf{J}^{\varepsilon} \wedge \mathbf{B}%
^{\varepsilon} }{n_e^{\varepsilon} },
\end{equation*}
the same integrability, locally uniformly in $\mathbf{v}$, holds true because these are
the same terms and thus they are bounded as for the electric field at the exception of $
\mathbf{v}\wedge{\mathbf{B}}^{\varepsilon}$ which is well defined thanks to
\eqref{estb}.

%-------------------------------------------------------------

\subsection{Space-time compactness}

%-------------------------------------------------------------

The bounds mentioned before allow us to extract subsequences that converge 
weakly and our task is now to prove several strong convergence results.
\newline

We first observe that, from Proposition \ref{propmagnetic} and \eqref{estb},
we may extract a subsequence, still denoted by $\mathbf{B}^\ep$ such that 
\begin{equation}
\mathbf{B}^\ep \underset{\varepsilon \to 0}{\longrightarrow} \mathbf{B }
\quad \textrm{strongly in } L^p(0,T; \mathrm{L}^q(\Omega)^3 \big), \; 1 \leq p
< 2 ,\; 1 \leq q < 6, \textrm{ and } 1\leq p <\infty, \; 1\leq q < 2.
\label{magnetic_cv}
\end{equation}
This is because (see Appendix \ref{ap:magnetic}) thanks to our a priori
estimates and Theorem \ref{theoBernardi}, the field  $\{ \mathbf{B}^{\varepsilon} \}$
is bounded in $\mathrm{L}^1(0,T,\mathrm{H}^1(\Omega)^3)$, and then, thanks
to Rellich Theorem, $\{ \mathbf{B}^{\varepsilon}(t,\cdot)\}$ is relatively
compact in $\mathrm{L}^6(\Omega)^3$ for all $t\in]0,T[$). Compactness in
time then follows from the Simon-Lions-Aubin lemma (see \cite{JS}) because, as
proved at the end of section \eqref{ssec:ape}, we control $L^r_t(L^p_x)$
integrability of the electric field and 
\begin{equation*}
\frac{\partial \mathbf{B}^{\varepsilon}}{\partial t} = \nabla \wedge \mathbf{%
E}^{\varepsilon}.
\end{equation*}

According to the kinetic averaging lemma recalled in Appendix \ref{ap:kal},
we know that after extraction of a subsequence, for every test function $%
\psi =\psi (\mathbf{v})$ compactly supported, there exists a weak limit $%
f^{\ast }$ of $f^{\varepsilon }$ such that
$$
\left\langle \psi f^{\varepsilon }\right\rangle \rightarrow \left\langle
\psi f^{\ast }\right\rangle \quad \mathrm{a.e.} 
\quad \mbox{ and }\quad 
\left\langle \psi \mathbf{v}f^{\varepsilon }\right\rangle 
\rightarrow
\left\langle \psi \mathbf{v}f^{\ast }\right\rangle \quad \mathrm{a.e.}
$$
Now, since $\displaystyle \int_{0}^T \int_{\Omega} |\left\langle \psi f^{\varepsilon
}\right\rangle |^{5/3}dxdt$ and $\displaystyle \int_{0}^T \int_{\Omega} |\left\langle \psi \mathbf{
v}f^{\varepsilon }\right\rangle |^{5/4}dxdt$ are uniformly bounded,
according to classical interpolation arguments we have for given exponents $
q$ and $r$ (with $1<q<5/3$ and $1<r<5/4)$
$$
\int_{0}^T \int_{\Omega}|\left\langle \psi f^{\varepsilon }\right\rangle
-\left\langle \psi f^{\ast }\right\rangle |^{q}dxdt 
\rightarrow 
0,
\quad
\mbox{ and }
\quad
\int_{0}^T \int_{\Omega}|\left\langle \psi \mathbf{v}f^{\varepsilon }\right\rangle
-\left\langle \psi \mathbf{v}f^{\ast }\right\rangle |^{r}dxdt 
\rightarrow 0.
$$
Therefore we get 
\begin{equation}
n_{I}^{\varepsilon }\underset{\varepsilon \rightarrow 0}{\longrightarrow }
n_{I},\textrm{ strongly in }\mathrm{L}^{q}(0,T;\Omega ),  \label{ni_cv}
\end{equation}
\begin{equation}
n_{I}\mathbf{u}_{I}^{\varepsilon }\underset{\varepsilon \rightarrow 0}{
\longrightarrow }n_{I}\mathbf{u}_{I}\quad \textrm{strongly in }\mathrm{L}
^{r}(0,T;\Omega )^{3}  \label{nui_cv}.
\end{equation}
Next for the electronic density we may extract a subsequence,
still denoted $n_{e}^{\varepsilon}$ such that, 
\begin{equation}
n_{e}^{\varepsilon}\underset{\varepsilon \rightarrow 0}{\longrightarrow }
n_{e}\quad \textrm{a.e. and strongly in }\mathrm{L}^{q}(0,T;\Omega )^{3}.
\label{ne_cv}
\end{equation}
This is a consequence of the bounds \eqref{propnep} and \eqref{elliptic1} and
the strong convergence (\ref{ni_cv}).

%-------------------------------------------------------------

\subsection{Passing to the limit}

%-------------------------------------------------------------

We recall that a weak (distributional) solution to \eqref{model} is defined
by testing against smooth test functions $\Psi(t,x)$, $\Phi(t,x,\mathbf{v})$
respectively for \eqref{model}-b) and \eqref{model}-c) and satisfying the
corresponding boundary conditions as in \eqref{BC}.

After integration by parts, we obtain respectively for \eqref{model}-b) and 
\eqref{model}-c), the definitions 
\begin{equation}
\begin{array}{rl}
\displaystyle - \iint_{(0,T) \times \Omega } \Big[ \frac{\partial \Psi(t,x)}{
\partial t} \mathbf{B}^\ep +\mathbf{E^{\varepsilon}}(t,x)  \cdot \nabla \wedge
\Psi(t,x)\Big] dt dx = \int_{\Omega } \mathbf{B}^{in,\varepsilon}(x,\mathbf{v})
\Psi(0,x) , & 
\end{array}
\label{weaksomf}
\end{equation}
and 
\begin{equation}
\begin{array}{rl}
\displaystyle - \iiint_{(0,T) \times \Omega \times \mathbb{R}^3} \Big[ \frac{
\partial \Phi(t,x,\mathbf{v})}{\partial t} & + \mathbf{v} \cdot \nabla_x \Phi(t,x,
\mathbf{v}) + \mathbf{F^{\varepsilon}}(t,x,\mathbf{v}) \cdot \nabla_\vv \Phi(t,x,
\mathbf{v}) \Big]f^\ep(t,x,\mathbf{v}) dt dx d\mathbf{v} \\ 
& \displaystyle = \iint_{\Omega \times \mathbb{R}^3} f^{in,\varepsilon}(x,
\mathbf{v}) \Phi(0,x,\mathbf{v}) .
\end{array}
\label{weaksolk}
\end{equation}
The elliptic equation \eqref{model}-a) is more standard and does not yield
difficulties, hence we do not consider it here. In order to conclude the
proof of Theorem \ref{main}, our purpose is to show that the limit as $
\varepsilon \to 0$ of the various unknowns $\mathbf{B}^\ep$, $\mathbf{
E^{\varepsilon}}(t,x)$, $\mathbf{F^{\varepsilon}}$ and $f^\ep$ still satisfy
these equalities. \newline

Additionally to the strong convergence results stated in \eqref{magnetic_cv}--\eqref{ne_cv}, 
 we are know other weak convergences for functions of interest. Firstly, we have 
\begin{equation*}
\mathbf{J^{\varepsilon}} \underset{\varepsilon \to 0}{\relbar\joinrel
\rightharpoonup} \mathbf{J} \quad \textrm{ in } \mathrm{L}^2((0,T)\times
\Omega)^3.
\end{equation*}

This is enough to pass to the limit weakly in the electric fields because
the nonlinear terms are always formed of either strongly convergent terms or 
$\mathbf{J^{\varepsilon}}$ multiplied by a term that converges strongly. As
a conclusion, we have 
\begin{equation*}
\mathbf{E^{\varepsilon}} \underset{\varepsilon \to 0}{\relbar\joinrel
\rightharpoonup} \mathbf{E} \quad \textrm{ in } \mathrm{L}^{20/11}\big(0,T;
\mathrm{L}^{30/29}( \Omega) \big)^3.
\end{equation*}

The same applies to the force field. However for later purposes, it is
better to use the splitting 
\begin{equation*}
\mathbf{F}^{\varepsilon }=\mathbf{E}^{0,\varepsilon }+\mathbf{v}\wedge {
\mathbf{B}^{\varepsilon}},
\end{equation*}
and to notice that, as before, 
\begin{equation*}
\mathbf{E}^{0,\varepsilon }\underset{\varepsilon \rightarrow 0}{\relbar
\joinrel\rightharpoonup }\mathbf{E}^{0}\quad \textrm{ in }\mathrm{L}^{20/11}
\big(0,T;\textrm{L}^{30/29}(\Omega )\big)^{3}
\end{equation*}
while the other term converges strongly. \newline

These observations allow us to pass to the weak limit in equations 
\eqref{model}-a) and \eqref{model}-b), on the electronic density and on the
magnetic field. \newline

For passing to the weak limit in the third equation \eqref{model}-c), let us
introduce two smooth test functions $\chi $ and $\varphi $ with compact
support respectively in $[0,T)\times \Omega $ and in $\mathbf{R}^{3}$ ; then
we have to deal with two terms. The first one which reads as 
\begin{equation*}
\int_{0}^{T}\int_{\Omega }\int_{\mathbb{R}^{3}}\mathbf{B}^{0,\varepsilon
}\wedge \mathbf{v}\cdot \frac{\partial \varphi }{\partial \mathbf{v}}(\mathbf{v}
)\chi (t,x)f^{\varepsilon }(t,x,\mathbf{v})dtdxd\mathbf{v},
\end{equation*}
can be treated as usual because the magnetic field converges strongly as
stated earlier.
The second one reads as follows 
\begin{equation*}
\int_{0}^{T}\int_{\Omega }\int_{\mathbb{R}^{3}}\mathbf{E}^{0,\varepsilon } \cdot
\frac{\partial \varphi }{\partial \mathbf{v}}(\mathbf{v})f^{\varepsilon
}(t,x)\chi (t,x)dtdxd\mathbf{v}=\int_{0}^{T}\int_{\Omega }\chi (t,x)\mathbf{E
}^{0,\varepsilon }\cdot \left\langle \psi f^{\varepsilon }\right\rangle
(t,x)dtdx\qquad
\end{equation*}
with $\psi _{i}=\frac{\partial \varphi }{\partial \mathbf{v}_{i}}$ ; we need
to use the kinetic averaging lemma which is stated in appendix \ref{ap:kal}.
It states that, after extraction,
velocity averages converge strongly
\begin{equation*}
\chi \left\langle \psi _{i}f^{\varepsilon }\right\rangle 
(t,x) \rightarrow \chi \left\langle \psi _{i}f^{\ast }\right\rangle
 \quad \text{in } \quad L^{q}([0,T]\times \Omega )
\end{equation*}
  for a given value of the index $q$, with $1\leq q<p$. Therefore they
converge almost everywhere and this is enough to pass to the weak-strong
limit in the term $\int_{0}^{T}\int_{\Omega }\chi \mathbf{E}^{0,\varepsilon
}\cdot \left\langle \psi f^{\varepsilon }\right\rangle dtdx$. \newline

This concludes the proof of the Theorem \ref{main}.

%-------------------------------------------------

\section{Non homogeneous magnetic boundary condition}

\label{section:binhom}

In view of applications to the physics of confined plasmas, it is worthwhile 
considering, instead of (\ref{BC})-(b), a non homogeneous boundary condition for the magnetic field
\begin{equation}  \label{B2}
\mathbf{n}_x \wedge \mathbf{B}(t,x)= \mathbf{n}_x \wedge \mathbf{B}_{\mathrm{
imp}} , \quad x\in \partial\Omega,
\end{equation}
The imposed magnetic field  $B_{\mathrm{imp}}(x)$ is given and can be a function of 
time as well eventhough we restrict our analysis to space dependency only. In
order to introduce this boundary condition in the problem, we assume that $
\mathbf{B}_{\mathrm{imp}}$ can be defined globally as a smooth function in the closure of $\Omega$. 
It turns out that a convenient regularity assumption is 
\begin{equation}  \label{eq:assump}
\| \mathbf{B}_{\mathrm{imp}} \|_{1,\infty }\leq C.
\end{equation}
We will write $\mathbf{B}=\mathbf{B}_{\mathrm{imp}}+\mathbf{B}_{\mathrm{pert}
} $ where $\mathbf{B}_{\mathrm{pert}}$ is the perturbation. Using these
notations the system (\ref{model})-(a,b,c) is rewritten as 
\begin{equation}  \label{model-2}
\left\{ 
\begin{array}{ll}
\displaystyle- \lambda ^{2}\Delta \ln n_{e}=n_{I}-n_{e}, & \quad (a)\\[3mm] 
\displaystyle \frac{\partial \mathbf{B}_{\mathrm{pert}}}{\partial t}
-\nabla \wedge \left( \frac{1}{n_{e}} n_{I}\, \mathbf{u}_{I}\wedge 
 \left( \mathbf{B}_{\mathrm{imp}}+\mathbf{B}_{\mathrm{pert}} \right)\right)
 + \nabla \wedge \left( \frac{1}{n_{e}}\left( \mathbf{J}_{\mathrm{imp}} 
 + \mathbf{J}_{\mathrm{pert}} \right) \wedge \left( \mathbf{B}_{\mathrm{imp}}
 +\mathbf{B}_{\mathrm{pert}} \right) \right) &  \\ 
~ \hfill +\nabla \wedge \left( \eta \left( \mathbf{J}_{\mathrm{imp}}+
\mathbf{J}_{\mathrm{pert}} \right)\right) =0, & \quad (b) \\[3mm]
{\displaystyle \frac{\partial {f}}{\partial {t}}+\mathbf{v}.\nabla f+ \frac{
\partial }{\partial \mathbf{v}}\Biggl[\left( (-\frac{{T_{\mathrm{e}}}}{{n_{e}
}} \nabla {n_{e}+}\frac{\left( \mathbf{J}_{\mathrm{imp}} + \mathbf{J}_{
\mathrm{pert}} \right)-n_{_{I}}\mathbf{u}_{I}}{n_{e}}\wedge \left( \mathbf{B}
_{\mathrm{imp}}+\mathbf{B}_{\mathrm{pert}} \right))+\mathbf{v}\wedge \left( 
\mathbf{B}_{\mathrm{imp}}+\mathbf{B}_{\mathrm{pert}} \right)\right) f\biggr]
=0}, & \quad (c).
\end{array}
\right.
\end{equation}
where the total current is $\mathbf{J}= \mathbf{J}_{\mathrm{imp}}+\mathbf{J}%
_{\mathrm{pert}}$ with 
\begin{equation*}
\mathbf{J}_{\mathrm{imp}}=\nabla \wedge \mathbf{B}_{\mathrm{imp}} 
\mbox{ and
} \mathbf{J}_{\mathrm{pert}}=\nabla \wedge \mathbf{B}_{\mathrm{pert}}.
\end{equation*}%
We now use the boundary conditions
\begin{equation}
\left\{ \begin{aligned} & \mathbf{n}_x\cdot \nabla n_e(t,x) =0 \quad & x\in
\partial\Omega, & \quad \quad(a)
\\ & \mathbf{n}_x \wedge \mathbf B(t,x)_{\rm
pert}=0 \quad & x\in \partial\Omega, & \quad \quad(b)
\\ & f (x, \mathbf v -2
(\mathbf v \cdot \mathbf{n}_x)\mathbf{n}_x) = f (x, \mathbf v), \quad &\mathbf v \in \mathbb R^3, \quad  x\in
\partial\Omega. & \quad \quad(c)
\end{aligned}\right.  \label{BC-2}
\end{equation}
Let us define the perturbed magnetic energy and the total perturbed energy 
\begin{equation*}
\mathcal{E}_{m}^{\mathrm{pert}}(t)=\frac{1}{2}\int_{\Omega }|\mathbf{B}_{
\mathrm{pert}}(t,x)|^{2}dx,
\end{equation*}
\begin{equation*}
\mathcal{E}_{\mathrm{tot}}^{\mathrm{pert }}= \mathcal{E}_{I}+\mathcal{E}
_{m}^{\mathrm{pert }}+\frac{\lambda ^{2}}{2}\int_{\Omega }\left\vert \nabla
_{x}\left( \ln n_{e}\right) \right\vert ^{2}dx+\int_{\Omega }\left( n_{e}\ln
n_{e}-n_{e}+1\right) dx \geq 0.
\end{equation*}

The energy balance is modified by the imposed magnetic field and we have 
\begin{prop}
Classical solutions to (\ref{model-2})--(\ref{BC-2}) satisfy the perturbed
energy dissipation relation 
\begin{equation*}
\frac{d}{dt}\mathcal{E}_{\mathrm{tot}}^{\mathrm{pert}} = -\int_{\Omega }\eta %
\Bigl| \mathbf{J}_{\mathrm{pert}}\Bigr|^{2}dx +S
\end{equation*}
where the source is 
\begin{equation*}
S=-\int_{\Omega }\eta \ \mathbf{J}_{\mathrm{imp}} \cdot \mathbf{J}_{\mathrm{%
pert}}dx -\int_{\Omega } \frac1{n_e} \mathbf{J}_{\mathrm{imp}} \wedge \left( 
\mathbf{B}_{\mathrm{imp}}+\mathbf{B}_{\mathrm{pert}} \right) \cdot \mathbf{J}%
_{\mathrm{pert}}dx + \int_\Omega \frac1{n_e} \mathbf{J}_{\mathrm{imp}}
\wedge \left( \mathbf{B}_{\mathrm{imp}}+\mathbf{B}_{\mathrm{pert}} \right)
\cdot n_I \mathbf{\ u}_{I} dx.
\end{equation*}
\end{prop}

\begin{proof}
Performing the same manipulations as in the proof of the energy identity
(\ref{estimationNRJ}), we obtain
$$
\frac{d}{dt}\mathcal{E}_{I}= 
\int_{\Omega }\left( \frac{n_{I}}{n_{e}}
\left( \mathbf{J}_{\rm imp} + \mathbf{J}_{\rm pert} \right)\wedge 
\left( \mathbf{B}_{\rm imp} + \mathbf{B}_{\rm pert} \right)%
-n_{I}T_{\mathrm{e}}\nabla \left( \ln n_{e}\right) \right) \cdot \mathbf{u}%
_{I}dx.
$$
Taking the scalar product of the magnetic equation against $\mathbf B_{\rm pert} $
and using the homogeneous boundary condition
(\ref{BC-2})-(b), 
we obtain 
$$
\frac{d}{dt}\mathcal{E}_{m}^{\rm pert  }= 
-\int_{\Omega }\frac{n_{I}}{n_{e}}\mathbf{u}%
_{I}\cdot 
\left
(\mathbf{J}_{\rm pert}\wedge \left( \mathbf{B}_{\rm imp} + \mathbf{B}_{\rm pert} \right)\right)dx
-\int_{\Omega }
\frac1{n_e} 
\left( \mathbf{J}_{\rm imp} +
\mathbf{J}_{\rm pert}
\right) \wedge \left( \mathbf B_{\rm imp}+\mathbf B_{\rm pert}
\right) \cdot  
\mathbf{J}_{\rm pert}dx
$$
$$
-\int_{\mathbb{R}^{3}}\eta \left( \mathbf{J}_{\rm imp} +
\mathbf{J}_{\rm pert}\right) \cdot \mathbf{J}_{\rm pert}dx
$$
$$
=-\int_{\Omega }\frac{n_{I}}{n_{e}}\mathbf{u}_{I}\cdot \left
(\mathbf{J}_{\rm pert}\wedge \left( \mathbf{B}_{\rm imp} + \mathbf{B}_{\rm pert} \right)\right)dx
-\int_{\Omega } \frac1{n_e}  \mathbf{J}_{\rm imp} +
 \wedge \left( \mathbf B_{\rm imp}+\mathbf B_{\rm pert}\right) \cdot  \mathbf{J}_{\rm pert}dx
-\int_{\mathbb{R}^{3}}\eta \left( \mathbf{J}_{\rm imp} +
\mathbf{J}_{\rm pert}\right)\cdot \mathbf{J}_{\rm pert}dx
$$
Therefore one gets 
$$
\frac{d}{dt}\left( \mathcal{E}_{I}+\mathcal{E}_{\rm tot}^{\rm pert}\right)=-
\int_{\Omega }T_{\mathrm{e}}\nabla \left( \ln n_{e}\right)  \cdot( n_I \mathbf{u}_{I})dx
-\int_{\Omega }\eta \Bigl| \mathbf{J}_{\rm pert}\Bigr|^{2}dx+S
$$
which is very similar to (\ref{eq:ttt3}).
The rest of the proof is unchanged.
\end{proof}

\begin{prop}
\label{prop33} There exists a constant $K>0$ depending only on the
initial data and on the constant in (\ref{eq:assump}) such that the
perturbed energy is bounded for all time $t< T^\star=\log\left(1+\frac{C}{ \|%
\mathbf{J}_{\mathrm{imp}}\|_\infty } \right)$.
\end{prop}

\begin{rem}
Our analysis of weak stability can easily be extended to this non-homogeneous boundary condition  for 
$t<T^\star$. This estimate expresses the interest of a good control on the
imposed current. Indeed the smaller is $\|\mathbf{J}_{\mathrm{imp}}\|_\infty $,
the greater $T^\star$ is.
\end{rem}

\begin{proof} (of proposition \ref{prop33}).
We observe that the new terms in $S$
are proportional to $ \mathbf{J}_{\rm imp}$.
Therefore two cases occur.
\\
\\
%\begin{description} \item[
\noindent {\bf First case: $ \mathbf{J}_{\rm imp}=0$.} This idealized
case might be encountered in Tokamaks: for example if 
$\mathbf{B}_{\rm imp}= F \nabla \theta$
where $\theta$ is the  toroidal angle and $F$ is a constant.
Physically it corresponds to an imposed exterior magnetic with
vanishing current \cite{sart}. The magnetic lines  of
$\mathbf{B}_{\rm imp}$ form a ring.
In this  case $S=0$, so the perturbed  energy dissipation
relation has the same form as (\ref{estimationNRJ}).
It turns out that  the perturbed energy is bounded for all times,
 which indeed corresponds to the claim since  $ \mathbf{J}_{\rm imp}=0$.
\\
\\
%\item[
\noindent {\bf Second case: $ \mathbf{J}_{\rm imp}\neq 0$.} 
This situation is  more realistic  in Tokamaks, 
since the magnetic lines form an helix or at least must be close
to helicoidal geometry \cite{sart}. This is the general case.

Let $\sigma>0$ be a given positive number. 
The energy identity writes also 
\begin{equation} \label{eq:spec1}
\frac{d}{dt}\left( e^{-\sigma t}\mathcal{E}_{\rm tot}^{\rm pert}\right)
=Re^{-\sigma t}, \quad
R=-\sigma \mathcal{E}_{\rm tot}^{\rm pert}
-\int_{\Omega }\eta \Bigl| \mathbf{J}_{\rm pert}\Bigr|^{2}dx +S.
\end{equation}
Our goal is to 
control  the source term $S$ in $R$ as much as possible by 
$-
\sigma \mathcal{E}_{\rm tot}^{\rm pert}-
\eta \| \mathbf{J}_{\rm pert} \|_2
$
and to control the remaining part with a Gronwall technique.

Using the assumption (\ref{eq:assump}), the fact that $\eta$ is constant
and the H\"older inequality between the conjugated spaces $L^5(\Omega)$
and $L^\frac54(\Omega)$, 
we can write
$$
\left| S\right|
\leq 
\left(
\alpha_1 \| \mathbf{J}_{\rm pert} \|_2
+\alpha_2 \| \mathbf{J}_{\rm pert} \|_2
\| \mathbf{B}_{\rm pert} \|_2
+\alpha_2 \| \mathbf{B}_{\rm imp}\|_5\| n_I 
\mathbf{\ u}_{I} \|_{\frac54 } +\alpha_2
\| \mathbf{B}_{\rm pert} \|_5 
\| n_I \mathbf{\ u}_{I} \|_{\frac54 }\right) \|
 \mathbf{J}_{\rm imp} \|_\infty 
$$
with $\alpha_1=\eta$ and $\alpha_2=\|n_e^{-1}\|_\infty$.
Let $\epsilon>0$ be an arbitrary  real number.
Then
$$
\alpha_1 \| \mathbf{J}_{\rm pert} \|_2
+
\alpha_2 \| \mathbf{J}_{\rm pert} \|_2
\| \mathbf{B}_{\rm pert} \|_2
\leq
\frac{\alpha_1}{2\epsilon}  
+\epsilon
\frac{\alpha_1+\alpha_2}2
 \| \mathbf{J}_{\rm pert} \|_2^2
+
\frac{\alpha_2}{2\epsilon} 
\| \mathbf{B}_{\rm pert} \|_2^2.
$$
The identity (\ref{eq:54}), written as
$
\| n_I 
\mathbf{\ u}_{I} \|_{\frac54 }\leq \alpha_4 \mathcal{E}_{I}^{\frac45}
$, implies successively the controls
$$
 \| \mathbf{B}_{\rm imp}\|_5\| n_I 
\mathbf{\ u}_{I} \|_{\frac54 }
\leq \alpha_3 \mathcal{E}_{I}^{\frac45}
$$ 
$$
\| \mathbf{B}_{\rm pert} \|_5
\| n_I 
\mathbf{\ u}_{I} \|_{\frac54 }
\leq \alpha_5
\left(
\| \mathbf{B}_{\rm pert} \|_2+
\|  \mathbf{J}_{\rm pert} \|_2
\right)
\mathcal{E}_{I}^{\frac45}
\leq 
\alpha_6 \left(
\mathcal{E}_{m}^{ \rm pert}\right)^\frac12
 \mathcal{E}_{I}^{\frac45}
+
\epsilon \frac{\alpha_5}2\|  \mathbf{J}_{\rm pert} \|_2^2+
\frac{\alpha_5}{2\epsilon }  \mathcal{E}_{I}^{\frac8{10}}.
$$
So the right hand side in (\ref{eq:spec1}) is bounded by
$$
R
\leq
\beta_1+ \beta_2  
\mathcal{E}_{I}^{\frac45}+
\beta_3
 \left(
\mathcal{E}_{m}^{ \rm pert}\right)^\frac12
 \mathcal{E}_{I}^{\frac45}
$$
for some constants $\beta_{1,2,3}$.
Since the electronic density is bounded and
$0\leq \mathcal{E}_{I}+ \mathcal{E}_{m}^{ \rm pert}
\leq \mathcal{E}_{\rm tot}^{\rm pert}$
by construction, we also have 
$$
R
\leq 
 \gamma_1+ \gamma_2
 \left( \mathcal{E}_{\rm tot}^{\rm pert} \right)^\frac45
+
\gamma_3 
 \left( \mathcal{E}_{\rm tot}^{\rm pert} \right)^\frac{13}{10}
$$
for some constants $\gamma_{1,2,3}$ which do not depend on time.
Since
$y^{\frac45}\leq 1+ y^{\frac{13}{10}}$ for positive $y$, 
 we get the more compact form
$$
R\leq \delta_1 + \delta_2 
 \left( \mathcal{E}_{\rm tot}^{\rm pert} \right)^\frac{13}{10}
$$
for some constants $\delta_{1,2}$ which do not depend on time.

We see that the right hand side is more than linear with respect to
$\mathcal{E}_{\rm tot}^{\rm pert}$ due to the power $\frac{13}{10}$.
As a consequence this inequality cannot  prove that 
$e^{-\sigma t}\mathcal{E}_{\rm tot}^{\rm pert}$ or
$\mathcal{E}_{\rm tot}^{\rm pert}$
 is bounded for all time.

Next we wish to obtain a evaluation of the time of existence with respect
to $\|\mathbf{J}_{\rm imp} \|_\infty$.
We set $u(t)=e^{-\sigma t}\mathcal{E}_{\rm tot}^{\rm pert}$. One 
can simplify the inequality as
$$
u'(t)\leq \left( \delta_1+\delta_2 u^{\frac{13}{10}}\right)
\|\mathbf{J}_{\rm imp}
\|_\infty e^{\frac{3\sigma}{10} t}.
$$
Rescaling of the time variable as 
$d\tau=\|\mathbf{J}_{\rm imp}
\|_\infty e^{\frac{3\sigma}{10} t}dt$, that is 
$$
\tau= \frac{10}3 \|\mathbf{J}_{\rm imp}
\|_\infty \left( e^{\frac{3\sigma}{10} t}-1\right),
$$
yields the inequality
$
\frac{d}{d\tau} u \leq \delta_1+\delta_2 u^{\frac{13}{10}}$.
It is finally convenient to define $v=\delta_1+\delta_2 u^{\frac{13}{10}}$
so that
$$
v'(t)=\delta_2 \frac{13}{10}u^{\frac{3}{10}}u'(t)
\mbox{ which yields }
v'(t)\leq
\delta_3 v^{\frac{3}{13}  }v= 
\delta_3  v^{\frac{16}{13}  }.
$$
Therefore
$
-\frac{d}{d\tau}v^{-\frac{3}{13}  }\leq \delta_4=\frac{3}{13}\delta_3.
$ which implies
$
v^{-\frac{3}{13}  }(0)
-
v^{-\frac{3}{13}  }(\tau)\leq \delta_4\tau$. It yields
$$
v^{\frac{3}{13}  }(\tau)\leq \frac1{v^{\frac{3}{13}  }(0)-\delta_4 \tau  }
$$
which is valid for $\tau<\tau^\star= \frac{ v^{\frac{3}{13}  }(0) }{\delta_4}$.
Going back to the time variable $t$, the solution is defined for
$
t<T^\star
$ where
$$
\frac{10}3 \|\mathbf{J}_{\rm imp}
\|_\infty \left( e^{\frac{3\sigma}{10} T^\star}-1 \right)=\tau^\star.
$$
The proof is complete.

%\end{description}

\end{proof}

%--------------------------------------------------------------------
\section{Construction of an approximate solution}
%--------------------------------------------------------------------
%--------------------------------------------------------------------

In order to complete our theory, we now detail how to use the so-called splitting strategy,
which is a constructive method, for the design
of   an approximate solution to the system
\begin{equation}  \label{model:const}
\left\{ 
\begin{array}{ll}
{\displaystyle-}\lambda ^{2}\Delta {\ln {n_{e}}={n_{I}-n_{e}}}, & \quad (a)
\\[3mm] 
{\displaystyle\frac{\partial \mathbf{B}}{\partial {t}}-\nabla \wedge }\left( 
\frac{1}{n_{e}}n_{I}\mathbf{\ u}_{I}\wedge \mathbf{B}\right) +{\nabla \wedge
(}\frac{1}{n_{e}}\mathbf{J}\wedge \mathbf{B)}{+\nabla \wedge \left( \eta
\nabla \wedge \mathbf{B}\right) =0}, & \quad (b) \\[3mm] 
{\displaystyle\frac{\partial {f}}{\partial {t}}+\mathbf{v}.\nabla f+ \frac{%
\partial }{\partial \mathbf{v}}\Biggl[\left( (-\frac{{T_{\mathrm{e}}}}{{n_{e}%
}} \nabla {n_{e}+}\frac{\mathbf{J}-n_{_{I}}\mathbf{u}_{I}}{n_{e}}\wedge 
\mathbf{\ B})+\mathbf{v}\wedge \mathbf{B}\right) f\biggr]=0}, & \quad (c) .%
\end{array}%
\right.
\end{equation}%
The idea is clearly inspired from numerical methods.
It consists of a convenient splitting strategy {\it \`a la }
Strang, together
with the linearization and freezing of certain coefficients {\it \`a la }
Temam.
The main point is to decompose the total system
in simpler  parts which are conceptually  easier to solve
or  easier to analyse,   preserving at the same time 
 the decay of the energy identity %balance
$$
\mathcal{E}_{tot}=
\frac{1}{2}\int_{\Omega }{\int_{\mathbb{R}^{3}}f(t,x,%
\mathbf{v})|\mathbf{v}}|^{2}{d\mathbf{v}}dx
+
\frac{1}{2}\int_{\Omega }|\mathbf{B}(t,x)|^{2}dx+
\frac{\lambda ^{2}}{2}%
\int_{\Omega }\left\vert \nabla _{x}\left( \ln n_{e}\right) \right\vert
^{2}dx+\int_{\Omega }\left( n_{e}\ln n_{e}-n_{e}+1\right) dx.
$$
%Such a procedure can also be used
%Mfor numerical algorithms.
%\subsection{Algorithm}
Let  $\Delta t >0$ be a time step 
which is ultimately destinated to tend to zero.
We consider that 
$$
f(t_k)\mbox{ and }\mathbf{B}(t_k)
$$
are known at the beginning of the time step
$t_k=k\Delta t$.
We restrict the presentation to the core  of the method.
This constructive method also provides additional insights into the 
mathematical structure of the model.

\subsection{Vlasov-Poisson}

One first solves during the time step $\Delta t$
$$
\left\{ 
\begin{array}{ll}
{\displaystyle-}\lambda ^{2}\Delta {\ln {n_{e}}={n_{I}-n_{e}}}, & \quad (a)
\\[3mm] 
{\displaystyle \frac{\partial \mathbf{B}}{\partial {t}}=0, }& \quad (b) \\
{\displaystyle\frac{\partial {f}}{\partial {t}}+\mathbf{v}.\nabla f+
 \frac{\partial }{\partial \mathbf{v}}
\Biggl[ -\frac{T_{\mathrm{e}}}{n_{e}} \nabla {n_{e}} f\biggr]=0}. & \quad (c) %
\end{array}%
\right.
$$
%avec $\frac{\partial \mathbf{B}}{\partial {t}}=0$.
This is a  non linear Vlasov-Poisson equation which can
be considered as standard even if we know
very little mathematical literature about it.
It is easy to show that regular solutions preserve
the energy.
%je consi\`ere que c'est bien pos\'e pendant le pas de temps
%$\Delta t$. 
%L'\'energie 
%est pr\'eserv\'ee (classique).be considered 
This procedure  defines a new solution
 $$
{f}^\star(t_k+\Delta t)
\mbox{ and }
\mathbf{B}^\star(t_k+\Delta t)=\mathbf{B}(t_k).
$$

\subsection{Magnetic part, first stage}

For convenience we split the magnetic part of the equations
 in
two stages, the first one which is fundamental, and the second
which deals with less involved terms.
The first stage writes 
$$
\left\{ 
\begin{array}{ll} 
{\displaystyle\frac{\partial \mathbf{B}}{\partial {t}}-\nabla \wedge }\left( 
\frac{1}{n_{e}}n_{I}\mathbf{\ u}_{I}\wedge
{  \mathbf{B}}_{\mbox{frozen}}
\right) +{\nabla \wedge
(}\frac{1}{n_{e}}\mathbf{J}\wedge \mathbf{B}_{\mbox{frozen}})
{+\nabla \wedge \left( \eta
\nabla \wedge \mathbf{B}\right) =0}, & \quad (b) \\[3mm] 
{\displaystyle\frac{\partial {f}}{\partial {t}}
+ \frac{\partial }
{\partial \mathbf{v}}\Biggl[\left( \frac{\mathbf{J}}{n_{e}}\wedge 
\mathbf{\ B}_{\mbox{frozen}})\right) f\biggr]=0}, & \quad (c).%
\end{array}%
\right.
$$ 
The initial data is provided by the previous step of the algorithm
$$
f^\square(t_k)={f}^\star(t_k+\Delta t)
\mbox{ et }
\mathbf{B}^\square(t_k)=\mathbf{B}^\star(t_k+\Delta t).
$$
The frozen magnetic field is  
$$
\mathbf{B}_{\mbox{frozen}}=\mathbf{B}^\star(t_k+\Delta t).
$$ 
This frozen field is constant in time during the whole
time step.
This trick was first introduced by Temam in the context
of 
trilinear forms and Navier-Stokes equations for magnetic equations
\cite{temam0,temam}, see also \cite{gerbeau}.

One notices that  equation b) is now a linear one,
even if  equation c) is still formerly non linear because it has
a $n_e$ dependence.
However an explicit procedure allows to compute the solution.
%Mais cette non lin\'earit\'e se d\'evisse sans trop de difficult\'es.
Indeed solutions of  equation  c)  are such that  
%implique apr\`es int\'egration en vitesse 
$$
\partial_t n_I=0.
$$
It means that  $n_I$ and  $n_e$  
are  frozen quantities
%sont aussi gel\'es
$$
n_I={n_I}^{\mbox{frozen}} \mbox{ et }n_e={n_e}^{\mbox{frozen}}.
$$
One has 
$$
\partial_t  n_{I}\mathbf{\ u}_{I}=  \mathbf{J} \wedge
\mathbf{d},
\quad \mathbf{d}=
%\left( 
\frac{{n_I}^{\mbox{frozen}} 
 \mathbf{\ B}_{\mbox{frozen}} }{ {n_e}^{\mbox{frozen}} }
%\right)
$$
which yields 
$$
 n_{I}\mathbf{\ u}_{I}=
 n_{I}\mathbf{\ u}_{I}(t_k)+\int_{t_k}^t 
\nabla \wedge \mathbf{B}(s)ds \wedge
\mathbf{d}.
$$
It shows  that $n_I\mathbf{\ u}_{I}$ integro-differential 
and linear 
with respect to $\mathbf{B}$.
%This equation is well posed under general assumptions.
%est int\'egro-diff\'erentiel
%par rapport \`a $\mathbf{B}$, mais surtout il est lin\'eaire.
Plugging this form of  $n_I\mathbf{\ u}_{I}$
 in  equation  b), 
we end up with a linear equation for
%on obtient une \'equation lin\'eaire
%en
  $\mathbf{B}$. 
This linear integro-differential
equation is  well posed under general assumptions.

%Je consid\`ere que cette \'equation est bien pos\'ee
%car je ne vois pas de raison de paniquer.

Once the magnetic field is computed, we can report
%
%Puis $\mathbf{B}$ \'etant calcul\'e, on reporte
the current 
 $\mathbf{J}$ in  equation  c) which
is now easily solved with the method of characteristics.
% que l'on r\'esout avec la m\'ethode
%des caract\'eristiques.
It is immediate that, due to the resistive operator,
 the energy decreases during this step.
%On v\'erifie sans peine que l'\'energie est dissip\'ee
%\`a cause du terme de r\'esisitivit\'e.
Since  $n_i$ is constant, the electronic energy is constant.
% pour montrer que l'\'energie 
%\'electronique ne bouge pas.

The solution at the end of this stage is referred to as 
 $$
{f}^\square(t_k+\Delta t)
\mbox{ and  }
\mathbf{B}^\square(t_k+\Delta t)=\mathbf{B}(t_k).
$$

\subsection{Magnetic part, second stage}

It remains to solve 
$$
\left\{ 
\begin{array}{ll} 
{\displaystyle\frac{\partial \mathbf{B}}{\partial {t}} =0}, & \quad (b) \\[3mm] 
{\displaystyle\frac{\partial {f}}{\partial {t}}
+ \frac{%
\partial }{\partial \mathbf{v}}\Biggl[\left( 
\frac{-n_{_{I}}\mathbf{u}_{I}}{n_{e}}\wedge 
\mathbf{\ B})+\mathbf{v}\wedge \mathbf{B}\right) f\biggr]=0}, & \quad (c) , 
\end{array}%
\right.
$$
with prescribed initial data
$$
f^\bullet(t_k)={f}^\square(t_k+\Delta t)
\mbox{ and }
\mathbf{B}^\bullet(t_k)=\mathbf{B}^\square(t_k+\Delta t).
$$
%En fait  le champs magn\'etique est gel\'e  
Since the magnetic field is frozen
$$
\mathbf{B}=\mathbf{B}_{\mbox{frozen}  },
$$
 equation c) is greatly simplified.
Once again
the ionic density  $n_I$ and the electronic density
 $n_e$ are frozen
$$
n_I={n_I}^{\mbox{frozen}} \mbox{ and }n_e={n_e}^{\mbox{frozen}}.
$$
A consequence is 
$$
\partial_t n_I \mathbf{u}_{I}=
n_{_{I}}\mathbf{u}_{I}\wedge \mathbf{d},\quad
\mathbf{d}=\left( -
\frac{ {n_I}^{\mbox{frozen}} }{ {n_e}^{\mbox{frozen}} }
+1\right) \mathbf{B}_{\mbox{frozen}  }.
$$
The solution of this  linear equation is immediate.  
Therefore $n_I \mathbf{u}_{I}$ is known.
And finally the method of characteristics
can be used to solve c).
% ce qui permet
%ensuite
%de r\'esoudre
%c) par la m\'ethode des caract\'eristiques.
The energy is preserved during this second magnetic stage.
%On v\'erifie sans peine que l'\'energie est constante pour cette phase.

\subsection{Iterations}

%Au final on a construit la solution \`a la fin du pas de temps
The previous procedure allows us to design an approximate solution
$$
f(t_k+\Delta t)={f}^\bullet(t_k+\Delta t)
\mbox{ and }
\mathbf{B}(t_k+\Delta t)=\mathbf{B}^\bullet(t_k+\Delta t)
$$
one time step after the other.
With this procedure the total energy decreases and we can apply our stability analysis for proving existence.
%et l'\'energie est pr\'eserv\'ee
%\`a chaque \'etape.

%Puis on it\`ere en $k$ pour construire une solution
%approch\'ee \`a tout instant.

%
%########## APPENDIX ###################################################
%-----------------------------------------------------------------------------------------------------------------------------
\appendix
%-----------------------------------------------------------------------------------------------------------------------------

%-----------------------------------------------------------------------------

\section{Control on moments of $f_I$}

\label{ap:moments} 
%-----------------------------------------------------------------------------

Several type of controls on velocity moments of $f_I$ are available, see 
\cite{BP_bams}. Here we recall one of the most fundamental control in $L^p$
spaces based on the kinetic energy.

\begin{lemme}
\label{ni53} Let $f \in L^{\infty}_{t,x,\mathbf{v}}((0,T) \times \Omega
\times \mathbb{R}^3) \cap L_t^{\infty}(0,T;L^1_{x,\mathbf{v}}(\Omega\times
\mathbb{R}^3, |\mathbf{v}|^2 dxd\mathbf{v})$. Define  $n_{I}$ by 
\eqref{defni} and $n_{I} \mathbf{u}_i$ by \eqref{defui}. Then $n_{I} \in L_t^{\infty}\big(0,T;L^{5/3}_{x}(\Omega)\big)$, $n_{I} \mathbf{u}_I \in
L_t^{\infty}\big(0,T;L^{5/4}_{x}(\Omega)\big)$ and we have for all $t\in[0,T]$, 
\begin{equation}\label{eq:43}
 \|n_{I}(t,\cdot)\|_{5/3} \leq C {\| f
(t,\cdot,\cdot) \|_{\infty}}^{2/5} \left(\int_{\Omega} \int_{\mathbb{R}^3} f
(t,x,\mathbf{v})|\mathbf{v}|^2d\mathbf{v }dx \right)^{3/5},
\end{equation}
\begin{equation}  \label{eq:54}
 \|n_{I}(t,\cdot) \mathbf{u}_I \|_{5/4} \leq
C^{\prime}{\| f (t,\cdot,\cdot) \|_{\infty}}^{1/5} \left(\int_{\Omega} \int_{
\mathbb{R}^3} f (t,x,\mathbf{v})|\mathbf{v}|^2d\mathbf{v }dx \right)^{4/5}.
\end{equation}
\end{lemme}

\noindent 
\begin{proof} 
We only recall the proof of the result on $n_{I}$. Let $R>0$. We have
\[ 
\begin{aligned}
 n_{I}(t,x)  & =\int_{|\vv|\leq R }f (t,x,\vv) d\vv +\int_{|\vv|\geq R }f (t,x,\vv) d\vv \\[3mm]
              &\leq  C R^3 \|f (t,\cdot,\cdot)\|_{\infty}+\frac{1}{R^2} \int_{|\vv|\leq R }f (t,x,\vv) |\vv|^2d\vv .
\end{aligned}
\]
Then by minimization over $R$ we get
\[n_{I}(t,x) \leq C \, \|f (t,\cdot,\cdot)\|_{\infty}^{2/5} \, \left(\int_{\R^3}f (t,x,\vv) |\vv|^2d\vv \right)^{3/5}\]
and after integration we obtain the claim.
\end{proof}

%---------------------------------------------------------------------

\section{Uniform lower bound on $n_e$}

\label{ap:elliptic} 
%---------------------------------------------------------------------

The purpose of this section is to prove several properties that we have used
throughout the paper for the elliptic equation 
\begin{equation*}
\left\{
\begin{aligned}
{\displaystyle-\lambda }^{2}{\Delta \ln {n_{e}}+n_{e}={n_{I}}}, \qquad & x
\in \Omega, \\[3mm]
\frac{\partial n_e}{\partial \nu} =0, \qquad & x \in \partial \Omega,%
\end{aligned}
\right.
\end{equation*}
with a right hand side data satisfying $n_{I} \geq 0$, $\int n_{I}
=M_{initial}>0$ and $n_{I} \in L^{5/3}$. \newline

We are going to prove the estimate

\begin{lemme}
\label{encadrementne} Let $n_{I} \in L^{\infty}(0,T;L^{5/3}\cap L^1(\Omega))$
and $n_e$ a strong solution to the above equation, then we have the
two-sided control 
\begin{equation}  \label{elliptic}
0 < K_-(\| n_{I} \|_{5/3}) \leq n_e \leq K_+(\| n_{I} \|_{5/3}),
\end{equation}
for some continuous positive functions $K_\pm (\cdot)$ with $K_+ >1$
increasing, $K_-$ decreasing.
\end{lemme}

We also recall that integration of the equation gives  the electric neutrality relation 
\begin{equation*}
\|n_e\|_{1} = \| n_{I}\|_{1} .
\end{equation*}

We can now explain why the total energy \eqref{def_energy} is well defined
for weak solutions and also at initial time. From the lower and upper bound
in \eqref{elliptic}, we conclude that $n_{e} \ln n_{e} \in L^1(\Omega)$.
Also, multiplying the equation by $\ln n_{e}$, we find 
\begin{equation*}
\lambda \int_\Omega \left|\nabla_{x} \ln n_{e} \right|^2 dx = \int_\Omega
(n_I -n _e) \ln n_{e} \leq 2 \ln K_+ \| n_{I}\|_{1}.
\end{equation*}

\noindent 
\begin{proof} Then, we argue in two steps. Firstly, we multiply the equation by $n_e^{2/3}$ and integrate by parts. The H\"older inequality gives
$$
\frac {2 \lambda^2}{3}  \int \frac{ |\nabla n_e|^2 }{n_e^{1/3}} + \int n_e^{5/3} = \int n_{I} n_e^{2/3} \leq \|n_{I}\|_{5/3} \; \| n_e \|^{2/3}_{5/3} ,
$$
from which we conclude the bound
$$
 \| n_e \|_{5/3} \leq \| n_{I} \|_{5/3} .
$$

Secondly we use the elliptic regularity theory to conclude that $\ln n_e  - \langle \ln n_e \rangle_\Omega \in W^{2,p}$, $1\leq p \leq 5/3$ and thus 
$$
 \ln n_e - \langle \ln n_e \rangle_\Omega \in L^q, \qquad \forall q>1, \qquad \frac 1 q= \frac 1p - \frac{2}{3},
$$
where $ \langle  \phi  \rangle_\Omega$ denotes the average of the $L^1(\Omega)$ function  $\phi$ over $\Omega$.
Finally, because $\frac 23 > \frac 35$ we conclude from the Morrey estimates \cite{evans} that
$$
\| \ln n_e - \langle \ln n_e \rangle_\Omega \|_\infty \leq C(\| n_{I} \|_{5/3}) .
$$
The result follows immediately thanks to the control of $\langle \ln n_e \rangle_\Omega$ through \eqref{elliptic}.
\end{proof} 

%---------------------------------------------------------------------

\section{Compactness of the magnetic field}

\label{ap:magnetic} 
%---------------------------------------------------------------------

We have also used the Sobolev injection for Maxwell equations and we recall it in this appendix. We introduce the
following  spaces
\begin{equation*}
\mathrm{H_{curl}}(\Omega)=\{\mathbf{b}\in \mathrm{L}^2(\Omega)^3 / \, \nabla
\wedge \mathbf{b} \in \mathrm{L}^2(\Omega)^3 \},
\end{equation*}
\begin{equation*}
\mathrm{H_{div}}(\Omega)= \{\mathbf{b}\in \mathrm{L}^2(\Omega)^3 / \, \nabla
\cdot\mathbf{b} \in \mathrm{L}^2(\Omega)^3 \},
\end{equation*}
\begin{equation}  \label{defXn}
X_N(\Omega)=\{ \mathbf{b} \in \mathrm{H_{curl}}(\Omega) \cap \mathrm{H_{div}}
(\Omega) \, / \quad \mathbf{b}\wedge \mathbf{n}=0 \quad \mathrm{on} \quad
\partial \Omega \}.
\end{equation}
We recall the following result (see \cite{DaLi})

\begin{theo}
\label{theoBernardi} Assume that the domain $\Omega$ is of class $\mathcal{C}
^{1,1}$. Then the space $X_N(\Omega)$ is continuously imbedded in $\mathrm{H}
^1(\Omega)^3$.
\end{theo}

%-----------------------------------------------------------------------------

\section{Kinetic averaging lemma}

\label{ap:kal} 
%-----------------------------------------------------------------------------

We recall here one result of the theory of averaging lemmas for kinetic
equations. When $f(t,x,\mathbf{v})$ is solution of a kinetic equation, it
cannot be more regular that the initial data or the right hand-side.
However, averages in velocity gain regularity. Recall that the macroscopic
quantity $\langle f\psi \rangle $ is defined as 
\begin{equation*}
\langle f\psi \rangle (t,x)=\int_{\mathbb{R}^{d}}f(t,x,\mathbf{v})\psi (
\mathbf{v})d\mathbf{v}
\end{equation*}%
where $\psi $ is a given function in $\mathcal{C}_{c}^{\infty }(\mathbb{R}
^{d})$ (i.e. smooth with compact support), the averaging lemmas aim at
proving compactness properties on $\langle f\psi \rangle $. The first
version of these averaging Lemma has been established by Golse, Lions
Perthame and Sentis \cite{GLPS} for $f\in L_{t,x,v}^{2}$ solution of the
equation $\partial _{t}f+\mathbf{v}\cdot \nabla _{x}f=S$ with $S\in L_{t,x,
\mathbf{v}}^{2}$: it was proved that locally in time $\langle f\psi \rangle
\in H_{t,x}^{1/2}$. This version has been then be extended (by complex
interpolation) for the $L^{p}$ framework, $1<p<\infty $ in \cite{DiLiM}.
Moreover, more complex versions have been proved by Di Perna, Lions and
Meyer \cite{DiLiM}, who treat the case where $S$ is the $k$-th
derivative in velocity with a fractional derivative in $x$ strictly less
than one.

We use here a version proved by Perthame and Sougadinis \cite{perthamesoug}
(the equality is the exponent is due to Bouchut \cite{bouchutgolse}) which
expresses an optimal gain of regularity (a full derivative)

%\mathrmbf
\begin{theo} Let  $1<q<\infty $ and $f,$ $\mathbf{g}
=(g_{1},...,g_{d}) $ belong to $\mathrm{L}_{t,x,\mathbf{v}}^{q}(
\mathbf{R}^{1+3+3})$  and satisfy 
\begin{equation}
\frac{\partial f}{\partial t}+\mathbf{v}\cdot \nabla _{x}f=\frac{\partial }{
\partial \mathbf{v}}.\big((\mathbb{I}-\Delta _{t,x,\mathbf{v}})^{1/2}\mathbf{
g}\big).
\end{equation}
Then $\langle f\psi \rangle \in \mathrm{L}_{t,x}^{q}(\mathbf{R}
^{1+3}) $\ and there exists $C(p,\psi )$  such that  
\begin{equation*}
\Vert \langle f\psi \rangle \Vert _{q}\leq C(q,\psi )\Vert f\Vert
_{q}^{1-\alpha }\Vert \mathbf{g}\Vert _{q}^{\alpha }
\end{equation*}
 for a positive exponent $\alpha \leq \frac{1}{2}\min (\frac{1}{q},1- \frac{1}{q})$.
\end{theo}

In fact this theorem also proves strong compactness, which is the way we use
it in our context. Consider a truncation function $\chi =\chi (t,x)$ with
supp$\chi \subset (0,T]\times \Omega $ which is fixed in this paragraph.
Define the sequence of functions $\mathbf{w}^{\varepsilon }$ as $\mathbf{w}
^{\varepsilon }=\chi \mathbf{F}^{\varepsilon }f^{\varepsilon }$ and $
z^{\varepsilon }=Zf^{\varepsilon },$ with 
$\displaystyle Z=\frac{\partial \chi }{\partial t
}+\mathbf{v}.\nabla _{x}\chi $. They satisfy in $\mathbf{R}^{1+3+3}$ 
\begin{equation}
\frac{\partial }{\partial t}(\chi f^{\varepsilon })+\mathbf{v}\cdot \nabla
_{x}(\chi f^{\varepsilon })=\frac{\partial }{\partial \mathbf{v}}.\mathbf{w}
^{\varepsilon }+z^{\varepsilon }.  \label{depar}
\end{equation}
We know that $\mathbf{w}^{\varepsilon }$ is bounded in $\mathrm{L}^{1}\cap 
\mathrm{L}_{t,x,\mathbf{v}}^{p}(\mathbf{R}^{1+3+3})$ for some $p>1$ (here $
p=30/29$) and $z^{\varepsilon }$ and $\chi f^{\varepsilon }$ are bounded in $
\mathrm{L}^{1}\cap \mathrm{L}_{t,x,\mathbf{v}}^{\infty }(\mathbf{R}
^{1+3+3}). $

We are going to prove the following result for an exponent $q$ (with $1<q<p$).

\begin{lemme} Consider, after extraction,  the weak limit $f^{\ast}$ in $\mathrm{L}_{t,x,\mathbf{v}}^{p}(
\mathbf{R}^{1+3+3})$ of the sequence $f^{\varepsilon}$ weakly  then
\begin{equation}
\rho _{\psi }^{\varepsilon }=\chi \langle f^{\varepsilon }\psi \rangle
\rightarrow \rho _{\psi }^{\ast }=\chi \langle f^{\ast }\psi \rangle ,\quad 
\mathrm{strongly\;in}\quad \mathrm{L}_{t,x,\mathbf{v}}^{q}(\mathbf{R}
^{1+3+3}).  \label{resu}
\end{equation}
\end{lemme}
\begin{proof}
 We first define 
\begin{equation*}
\mathbf{g}^{\varepsilon }=(\mathbb{I}-\Delta _{x,t,v})^{-1/2}\mathbf{w}.
^{\varepsilon }
\end{equation*}
By regularizing effects the family $\mathbf{g}^{\varepsilon }$\ is
compact in $\mathrm{L}_{t,x,\mathbf{v}}^{1}(\mathbf{R}^{1+3+3})$ : indeed $\mathbf{g}^{\varepsilon }$
is given by a convolution product between $\mathbf{w}^{\varepsilon }$ and the fundamental solution of 
$(\mathbb{I}-\Delta _{x,t,v})^{-1/2}$. Moreover $\mathbf{g}^{\varepsilon }$ 
is bounded in $\mathrm{L}_{t,x,\mathbf{v}}^{p}(\mathbf{R}^{1+3+3})$ and then, by
interpolation argument, it is compact in $\mathrm{L}_{t,x,\mathbf{v}}^{q}(
\mathbf{R}^{1+3+3})$ for some $q$ (with $1<q<p$).
Since the family $\{\chi f^{\varepsilon }\}$ is uniformly bounded in $\mathrm{
L}_{t,x,\mathbf{v}}^{q}$ as well as $z^{\varepsilon },$ the theorem recalled
above (the term $z^{\varepsilon }$ can be written as a $\mathbf{v}$
-derivative without loss of generality) applied to $\chi f^{\varepsilon }$
allows to claim that the following bound holds 
\begin{equation*}
\Vert \rho _{\psi }^{\varepsilon }\Vert _{q}\leq C(q,\psi )C\Vert \mathbf{g}
^{\varepsilon }\Vert _{q}^{\alpha }.
\end{equation*}
Then, there exists a subsequence $f^{\varepsilon },\mathbf{w}^{\varepsilon } 
$ and functions $f^{\ast },\rho _{\psi }^{\ast },\mathbf{g}^{\ast },\mathbf{w
}^{\ast }$ such that 
\begin{equation*}
\begin{cases}
f^{\varepsilon }\rightharpoonup f^{\ast },\qquad \mathbf{w}^{\varepsilon
}\rightharpoonup \mathbf{w}^{\ast }\quad \mathrm{weakly\;in}\quad \mathrm{L}
_{t,x,\mathbf{v}}^{p}(\mathbf{R}^{1+3+3}), \\[3mm]
\rho _{\psi }^{\varepsilon }\rightharpoonup \rho _{\psi }^{\ast }\quad 
\mathrm{weakly\;in}\quad \mathrm{L}_{t,x,\mathbf{v}}^{q}(\mathbf{R}^{1+3+3}),
\\[3mm]
\mathbf{g}^{\varepsilon }\rightarrow \mathbf{g}^{\ast }\quad \mathrm{\
strongly\;in}\quad \mathrm{L}_{t,x,\mathbf{v}}^{q}(\mathbf{R}^{1+3+3}).
\end{cases}
\end{equation*}
and $\mathbf{g}^{\ast }=(\mathbb{I}-\Delta _{x,t,v})^{-1/2}\mathbf{w}^{\ast
}$.
Of course we have also, passing to the weak limit, 
\begin{eqnarray*}
\left(\frac{\partial }{\partial t}+\mathbf{v}\cdot \nabla _{x})(\chi f^{\ast }\right)
&=&\frac{\partial }{\partial \mathbf{v}}\cdot\mathbf{w}^{\ast }+Zf^{\ast } \\
&=&\frac{\partial }{\partial \mathbf{v}}\cdot((\mathbb{I}-\Delta _{x,t,v})^{1/2}
\mathbf{g}^{\ast })+Zf^{\ast }.
\end{eqnarray*}
We combine this with (\ref{depar}) and we get 
\begin{equation*}
(\frac{\partial }{\partial t}+\mathbf{v}\cdot \nabla _{x})[\chi
f^{\varepsilon }-\chi f^{\ast }]=\frac{\partial }{\partial \mathbf{v}}\cdot((
\mathbb{I}-\Delta _{x,t,v})^{1/2}(\mathbf{g}^{\varepsilon }-\mathbf{g}^{\ast
}))+Z(f^{\varepsilon }-f^{\ast }).
\end{equation*}
And according to the previous theorem (the term $Z(f^{\varepsilon }-f^{\ast
})$ can be written as a $\mathbf{v}$-derivative also) once again we see that 
\begin{equation*}
\Vert \chi \langle f^{\varepsilon }\psi \rangle -\chi \langle f^{\ast }\psi
\rangle \Vert _{q}\leq C(q,\psi )\cdot \Vert \mathbf{g}^{\varepsilon }-\mathbf{g}
^{\ast }\Vert _{q}^{\alpha }\;\Vert f^{\varepsilon }-f^{\ast }\Vert
_{q}^{(1-\alpha )}.
\end{equation*}
That is to say, the property (\ref{resu}) holds.
\end{proof}

%-------------

%-----------------------------------------------------------------------------

%----------------------------------------------------------------------------------------------------------------------------------------------------------------------------

\end{document}